\documentclass[12pt]{amsart}

\usepackage{bm}
\usepackage{todonotes}
\usepackage{latexsym}
\usepackage{amsmath}
\usepackage{amsfonts}
\usepackage{amssymb}
\usepackage{geometry} 
\usepackage{multirow}
\usepackage{graphicx,color}
\usepackage{epstopdf}
\usepackage{caption}
\usepackage{diagbox}
\usepackage{subfig}
\usepackage{bm}
\usepackage{makecell}
\usepackage{array}
\usepackage{algorithm}
\usepackage{algorithmic}
\usepackage{tikz}
\usetikzlibrary{arrows.meta}
\usepackage{latexsym}
\usetikzlibrary{shapes.geometric, arrows}
\tikzstyle{startstop} = [rectangle, rounded corners, minimum width=3cm, minimum height=1cm,text centered, text width=5.5cm, draw=black]
\tikzstyle{arrow} = [thick,->,>=stealth]
\graphicspath{{figures/}}
\allowdisplaybreaks[4]

\newtheorem{theorem}{Theorem}[section]
\newtheorem{lemma}[theorem]{Lemma}

\theoremstyle{definition}
\newtheorem{definition}[theorem]{Definition}
\newtheorem{example}[theorem]{Example}

\theoremstyle{remark}
\newtheorem{remark}[theorem]{Remark}

\numberwithin{equation}{section}

\def\bz{\bm{\zeta}}

\def\bM{\bm{M}}
\def\bn{\bm{n}}
\def\bA{\bm{A}}
\def\bS{\bm{S}}
\def\bP{\bm{P}}
\def\bu{\bm{u}}
\def\bv{\bm{v}}
\def\Ho{H^1_0(\Omega)}
\def\O{\Omega}

\def\LT{{L_2(\O)}}

\def\cB{\mathcal{B}}
\def\cE{\mathcal{E}}
\def\cT{\mathcal{T}}

\def\HO{{H^1(\O)}}
\def\H2{{H^2(\O)}}

\def\eps{\varepsilon}

\def\Heh{{H^1_\eps(\O;\mathcal{T}_h)}}
\def\Heb{{H^1_{\eps,\beta}(\O)}}

\newcommand{\trinorm}[1]{%
  |\mkern-1.5mu|\mkern-1.5mu|
   #1
  |\mkern-1.5mu|\mkern-1.5mu|
}
\DeclareMathOperator*{\argmin}{argmin}

\begin{document}

\title[DG-MG for An Elliptic Optimal Control Problem]{Multigrid preconditioning for discontinuous Galerkin discretizations of an elliptic optimal control problem with a convection-dominated state equation}

\author{Sijing Liu}
\address{The Institute for Computational and Experimental Research in Mathematics\\
 Brown University\\
 Providence, RI\\
 USA}
 \email{sijing\_liu@brown.edu}
 \author{Valeria Simoncini}
 \address{Dipartimento di Matematica and (AM)$^2$, Alma Mater Studiorum - Università di Bologna, 40126 Bologna, and IMATI-CNR, Pavia,  Italy}
 \email{valeria.simoncini@unibo.it}

\date{\today}

\begin{abstract}
    We consider discontinuous Galerkin methods for an elliptic distributed optimal control problem constrained by a convection-dominated problem. We prove global optimal convergence rates using an inf-sup condition, with the diffusion parameter $\eps$ and regularization parameter $\beta$ explicitly tracked. We then propose a multilevel preconditioner based on downwind ordering to solve the discretized system. The preconditioner only requires two approximate solves of single convection-dominated equations using multigrid methods. Moreover, for the strongly convection-dominated case, only two sweeps of block Gauss-Seidel iterations are needed. We also derive a simple bound indicating the role played by the multigrid preconditioner. Numerical results are shown to support our findings.
\end{abstract}

\maketitle

\section{Introduction}

We consider the following elliptic optimal control problem.
Let $\Omega$ be a bounded convex polygonal domain in $\mathbb{R}^2$, $y_d\in \LT$ and $\beta$ be a positive constant. Find
\begin{equation}\label{optcon}
(\bar{y},\bar{u})=\argmin_{(y,u)}\left [ \frac{1}{2}\|y-y_d\|^2_{\LT}+\frac{\beta}{2}\|u\|^2_{\LT}\right],
\end{equation} 
where $(y,u)$ belongs to $H^1_0(\Omega)\times \LT$ and such that
\begin{equation}\label{eq:stateeq}
a(y,v)=\int_{\Omega}uv \ dx \quad \forall v\in H^1_0(\Omega).
\end{equation}
Here the bilinear form $a(\cdot,\cdot)$ is defined as
\begin{equation}\label{eq:abilinear}
  a(y,v)=\eps\int_{\Omega} \nabla y\cdot \nabla v\ dx+\int_{\Omega} (\bm{\zeta}\cdot\nabla y) v\ dx+\int_{\Omega} \gamma yv\ dx,
\end{equation}
where $\eps>0$, the vector field $\bm{\zeta}\in [W^{1,\infty}(\Omega)]^2$ and the function $\gamma\in L_{\infty}(\Omega)$ is nonnegative. We assume 
\begin{equation}\label{eq:advassump}
    \gamma-\frac12\nabla\cdot\bz\ge\gamma_0>0\quad \text{a.e.}\ \text{in}\ \Omega ,
\end{equation}
so that the problem \eqref{eq:stateeq} is well-posed. We mainly focus on the convection-dominated regime, 
namely, the case where $\eps\ll\|\bm{\zeta}\|_{0,\infty}:=\|\bm{\zeta}\|_{[L^\infty(\Omega)]^2}$.  

\begin{remark}
Throughout the paper we will follow the standard notation for differential operators, function spaces and norms that can be found for example in \cite{BS,Ciarlet,di2011mathematical}.
\end{remark}

It is well-known, see, e.g. \cite{Lions, Tro}, that the solution of \eqref{optcon}-\eqref{eq:stateeq} is characterized by
\begin{subequations}\label{eq:sp}
\begin{alignat}{3}
a(q,\bar{p})&=(\bar{y}-y_d,q)_\LT \quad &&\forall q\in H^1_0(\Omega),\\
\bar{p}+\beta\bar{u}&=0,\label{eq:sp2}\\
a(\bar{y},z)&=(\bar{u},z)_\LT  \quad &&\forall z\in H^1_0(\Omega),
\end{alignat}
\end{subequations}
where $\bar{p}$ is the adjoint state.
After eliminating $\bar{u}$ (cf. \cite{hinze2005variational}), we arrive at the saddle point problem

\begin{subequations}\label{eq:osp}
\begin{alignat}{2}
(\bar{p},z)_\LT+\beta a(\bar{y},z)&=0  \quad &&\forall z \in H^1_0(\Omega),\label{eq:osp2}\\
a(q,\bar{p})-(\bar{y},q)_\LT&=-(y_d,q)_\LT \quad &&\forall q\in H^1_0(\Omega).\label{eq:ospdual}
\end{alignat}
\end{subequations}

\vskip 0.1in
Note that the system \eqref{eq:osp} is unbalanced with respect to $\beta$
 since it only appears in \eqref{eq:osp2}.  This can be remedied by the following change of variables:
\begin{equation}\label{eq:CV}
\bar{p}=\beta^{\frac{1}{4}}\tilde{p}\quad\text{and}\quad
\bar{y}=\beta^{-\frac{1}{4}}\tilde{y}.
\end{equation}
 The resulting saddle point problem is
\begin{subequations}\label{subeq:BSPP}
\begin{alignat}{3}
(\tilde{p},{z})_\LT+\beta^{\frac{1}{2}}a(\tilde{y}, z)_\LT&=0
 &\qquad& \forall\,{z}\in \Ho,\label{eq:BSPP2}\\
 \beta^{\frac{1}{2}}a(q,\tilde{p})_\LT-(\tilde{y},q)_\LT
 &=-\beta^{\frac{1}{4}}(y_d,q)_\LT
  &\qquad&\forall \,{q}\in \Ho.\label{eq:BSPP1}
\end{alignat}
\end{subequations}

\subsection{Difficulties of designing and analyzing numerical methods for \eqref{subeq:BSPP}}

There are several difficulties regarding designing and analyzing numerical methods for \eqref{subeq:BSPP}. First, standard Galerkin methods for convection-dominated problems are known to be unstable and produce oscillations near the outflow boundary. Therefore, stabilization techniques are necessary to obtain any meaningful solutions of such problems. Moreover, the saddle point problem \eqref{subeq:BSPP} consists of a forward problem \eqref{eq:BSPP2} with convection field $\bm{\zeta}$ and a dual problem \eqref{eq:BSPP1} with convection field $-\bm{\zeta}$. This distinct feature in optimal control problems plays an important role in designing stable and accurate numerical methods. In fact, it has been shown in \cite{leykekhman2012local,heinkenschloss2010local} that the opposite convection fields in optimal control problems are nontrivial to handle. The boundary layers in both directions will propagate into the interior domain even if a stabilization technique is used (cf. \cite{heinkenschloss2010local}). This phenomenon is essentially different from the behaviors of boundary layers in single convection-dominated equations, in which it is well-known that the boundary layer will not propagate into the interior of the domain if proper stabilization techniques are utilized. One significant finding in \cite{leykekhman2012local} is that the weak treatment of the boundary conditions prevents the oscillations near the boundary layers from propagating into interior domain where the solution is smooth. Note that this can be done by using Nitsche's methods \cite{nitsche1971variationsprinzip} or discontinuous Galerkin methods \cite{leykekhman2012local,yucel2012distributed}.

\subsection{Difficulties of designing efficient solvers for \eqref{subeq:BSPP}}
Designing fast iterative solvers for the resulting discretized system from \eqref{subeq:BSPP} is nontrivial, especially in the convection-dominated regime. In this work, we focus on designing multigrid methods. For single convection-diffusion-reaction equations, it is well known that designing robust multigrid methods is difficult (see Section \ref{sec:mgdcr}). Designing and analyzing multigrid methods for saddle point problems like \eqref{subeq:BSPP} is even more challenging, and proper preconditioners must be devised.
In \cite{BLS,brenner2017multigrid,brenner2020multigrid,BOS,liu2024robust}, the authors designed a class of block-diagonal preconditioners and performed rigorous analyses of the multigrid methods that converge in the energy norm. Other approaches can be found in \cite{Scho,Takacs,simon2009schwarz,schoberl2007symmetric} and the references therein. However, almost all the preconditioners deteriorate in the convection-dominated regime. We refer to \cite{mardal2022robust,pearson2011fast} for a known robust preconditioner in the convection-dominated regime which is based on the Schur complement.
\
\\

Our contributions in this paper are two-fold. First, we propose and analyze an upwind discontinuous Galerkin (DG) method for solving \eqref{subeq:BSPP} where the diffusion parameter $\eps$ and regularization parameter $\beta$ are explicitly tracked. We show that the DG methods are optimal, for fixed $\beta$, in the sense of
\begin{equation}\label{eq:opesti}
    \|p-p_h\|_{1,\eps}+\|y-y_h\|_{1,\eps}\le\left\{\begin{array}{ll}
        O(h)\quad&\text{if}\ \eqref{eq:stateeq}\ \text{is diffusion-dominated,}\\
        \\
        O(h^\frac32)\quad&\text{if}\ \eqref{eq:stateeq}\ \text{is convection-dominated,}\\
        \\
        O(h^2)\quad&\text{if}\ \eqref{eq:stateeq}\ \text{is reaction-dominated,}
        \end{array}\right.
\end{equation}
where the norm $\|\cdot\|_{1,\eps}$ is defined in \eqref{eq:hbhnorm} and $h$ is the meshsize of the triangulation. Here $(p,y)$ are solutions to \eqref{eq:regu} and $(p_h,y_h)$ are solutions to \eqref{eq:dgprob}. Our analysis is based on an inf-sup condition \cite{brenner2020multigrid} and a crucial boundedness result in \cite{di2011mathematical}. Note that the control is not explicitly discretized, instead, we eliminate the control using the adjoint state \cite{hinze2005variational}, and hence we have a saddle point problem involving the state and the adjoint state. This technique is well-known, and it can be found, for instance, in \cite{brenner2020multigrid,gong2022optimal,gaspoz2019quasi}. We would like to point out that our DG methods are identical to those in \cite{leykekhman2012investigation,leykekhman2012local,yucel2012distributed}, where similar estimates to \eqref{eq:opesti} were derived in \cite{leykekhman2012investigation}. However, in our analysis, we do not decouple the state and the adjoint state by using intermediate problems, instead, we utilize the inf-sup condition and analyze the state and the adjoint state simultaneously, which is different from the one in  \cite{leykekhman2012investigation}.

Secondly, we design an efficient preconditioner to solve the discretized system. We combine the block-structured preconditioner by Pearson and Wathen \cite{pearson2011fast} with downwind ordering multigrid methods \cite{gopalakrishnan2003multilevel} to construct a highly efficient preconditioner. There are two advantages to combining DG methods with the preconditioner in \cite{pearson2011fast}. First, the mass matrix of DG methods is block-diagonal, allowing the inverse of the mass matrix to be computed exactly. Second, the downwind ordering technique makes the multigrid methods with block Gauss-Seidel iteration almost an exact solver as $\eps \rightarrow 0$. In particular, as $\eps \rightarrow 0$, a single sweep of block Gauss-Seidel is almost an exact solver, eliminating the need for multigrid cycles in those cases. Overall, with theses techniques, the implementation of our preconditioner is extremely efficient in the convection-dominated regime, which only consists of two multigrid solves of single convection-diffusion-reaction equations. In terms of the quality of the preconditioner, we provide a bound of the distance between the approximate preconditioner and the ideal preconditioner, which justifies the efficiency of our preconditioner. Note that we mainly focus on the case where $\eps\rightarrow0$ in this work, i.e, the convection-dominated case, instead of the case where $\beta\rightarrow0$, which is in contrast to \cite{pearson2011fast,brenner2020multigrid}. Nonetheless, numerical results in Section \ref{sec:numerics} indicate that our preconditioner is also robust when $\beta\rightarrow0$.

The rest of the paper is organized as follows.  In Section \ref{sec:conprob}, we discuss the properties of the continuous problem \eqref{subeq:BSPP} and establish its well-posedness. In Section \ref{sec:dgfem}, we introduce the DG methods and derive the inf-sup condition as well as an important boundedness result. In Section \ref{sec:convanalysis}, we establish concrete error estimates for the DG methods in the convection-dominated regime. We then propose a block preconditioner in Section \ref{sec:mgprecon}, where a crucial downwind ordering multigrid method with block Gauss-Seidel smoothers is introduced. A simple estimate is also derived in Section~\ref{sec:mgprecon} to illustrate the quality of our preconditioner. Finally, we provide some numerical results in Section \ref{sec:numerics} and end with some concluding remarks in Section \ref{sec:cncldremarks}.

Throughout this paper, we use $C$
 (with or without subscripts) to denote a generic positive
 constant that is independent of any mesh
 parameter, $\beta$ and $\eps$, unless otherwise stated.
  Also to avoid the proliferation of constants, we use the
   notation $A\lesssim B$ (or $A\gtrsim B$) to
  represent $A\leq \text{(constant)}B$. The notation $A\approx B$ is equivalent to
  $A\lesssim B$ and $B\lesssim A$.

\section{Continuous Problem}\label{sec:conprob}

We rewrite \eqref{subeq:BSPP} in a concise form
\begin{equation}\label{eq:conciseform}
\mathcal{B}((\tilde{p},\tilde{y}),(q,z))=-\beta^{\frac14}(y_d,q)_\LT \quad \quad \forall (q,z)\in H^1_0(\Omega)\times H^1_0(\Omega),
\end{equation}
where
\begin{equation}\label{bilinear}
\mathcal{B}((p,y),(q,z))=\beta^\frac12a(q,p)-(y,q)_\LT+(p,z)_\LT+\beta^\frac12 a(y,z).
\end{equation}

Let $\|p\|_{H^1_{\eps,\beta}(\Omega)}$ be defined by
\begin{equation}
\|p\|^2_{H^1_{\eps,\beta}(\Omega)}=\beta^\frac12\Big(\eps|p|^2_{H^1(\Omega)}+\|p\|^2_{\LT}\Big)+\|p\|^2_{\LT}.
\end{equation}
We have the following lemmas regarding the bilinear form $\cB$ with respect to the norm $\|\cdot\|_\Heb$.
\begin{lemma}\label{lem:continuous}
We have
\begin{equation}\label{cont}
\mathcal{B}((p,y),(q,z))\lesssim \frac{1}{\sqrt{\eps}}(\|p\|^2_\Heb+\|y\|^2_\Heb)^{\frac{1}{2}}(\|q\|^2_\Heb+\|z\|^2_\Heb)^{\frac{1}{2}}
\end{equation}
for any $(p,y), (q,z)\in H^1_0(\Omega)\times H^1_0(\Omega).$
\end{lemma}

\begin{proof}
It follows from integration by parts and Cauchy-Schwarz inequality (cf. \cite[Chapter 9]{knabner2004numerical}) that 
\begin{equation}
    \begin{aligned}
        \mathcal{B}((p,y),(q,z)) &\le\beta^\frac12\eps|p|_{H^1(\Omega)}|q|_{H^1(\Omega)}+\beta^\frac12\|\bm{\zeta}\|_{0,\infty}\|p\|_{\LT}|q|_{H^1(\Omega)}\\
        &\quad\quad+\beta^\frac12(|\bm{\zeta}|_{1,\infty}+\|\gamma\|_{\infty})\|q\|_{\LT}\|p\|_{\LT}\\
&\quad+\beta^\frac12\eps|y|_{H^1(\Omega)}|z|_{H^1(\Omega)}+\beta^\frac12\|\bm{\zeta}\|_{0,\infty}\|y\|_{\LT}|z|_{H^1(\Omega)}\\
&\quad\quad+\beta^\frac12(|\bm{\zeta}|_{1,\infty}+\|\gamma\|_{\infty})\|y\|_{\LT}\|z\|_{\LT}\\
&\quad+\|q\|_{L^2(\Omega)}\|y\|_{L^2(\Omega)}+\|p\|_{L^2(\Omega)}\|z\|_{L^2(\Omega)}\\
&\lesssim (\|p\|^2_\Heb+\|y\|^2_\Heb)^{\frac{1}{2}}\\
&\quad\quad\times(\beta^\frac12\|q\|^2_\HO+\|q\|^2_\LT+\beta^\frac12\|z\|^2_\HO+\|z\|^2_\LT)^{\frac{1}{2}}\\
&\lesssim \frac{1}{\sqrt{\eps}}(\|p\|^2_\Heb+\|y\|^2_\Heb)^{\frac{1}{2}}(\|q\|^2_\Heb+\|z\|^2_\Heb)^{\frac{1}{2}}.
    \end{aligned}
\end{equation}
\end{proof}

\begin{lemma}\label{lem:l2}
We have
\begin{equation}\label{eq:binfsup}
\sup_{(q,z)\in H^1_0(\Omega)\times H^1_0(\Omega)}\frac{\mathcal{B}((p,y),(q,z))}{(\|q\|^2_\Heb+\|z\|^2_\Heb)^\frac{1}{2}}\ge 2^{-\frac{1}{2}}(\|p\|^2_\Heb+\|y\|^2_\Heb)^\frac{1}{2}
\end{equation}
for any $(p,y) \in H^1_0(\Omega)\times H^1_0(\Omega).$
\end{lemma}

\begin{proof}
Given $(p,y)$, we take $q=p-y,\ z=p+y$. We have
\begin{equation}\label{eq:bcoer}
\begin{aligned}
\mathcal{B}((p,y),(q,z))&=\beta^\frac12a(p,p)+(p,p)+(y,y)+\beta^\frac12a(y,y)\\
&\gtrsim\|p\|^2_\Heb+\|y\|^2_\Heb
\end{aligned}
\end{equation}
where we use the fact $\gamma-\frac12\nabla\cdot\bm{\zeta}\ge \gamma_0\ge0 $ a.e. in $\Omega$.
Also, due to the parallelogram law, we have,
\begin{equation}\label{eq:paral}
(\|q\|^2_\Heb+\|z\|^2_\Heb)^\frac{1}{2}=2^\frac{1}{2}(\|p\|^2_\Heb+\|y\|^2_\Heb)^\frac{1}{2}.
\end{equation}
Combining \eqref{eq:bcoer} and \eqref{eq:paral}, we immediately obtain \eqref{eq:binfsup}. 
\end{proof}

\begin{remark}
    According to  standard saddle point theory \cite{Bab,Brezzi}, Lemma \ref{lem:continuous} and Lemma \ref{lem:l2} guarantee the well-posedness of the problem \eqref{eq:conciseform}.
\end{remark}

For the sake of generality, we shall also consider the following more general problem. Let $(p,y)\in H^1_0(\Omega)\times H^1_0(\Omega)$ satisfies
\begin{equation}\label{eq:regu}
\mathcal{B}((p,y),(q,z))=(f,q)_\LT+(g,z)_\LT\quad \quad \forall (q,z)\in H^1_0(\Omega)\times H^1_0(\Omega),
\end{equation}
where $(f,g)\in \LT\times \LT$ and $\mathcal{B}$ is defined in \eqref{bilinear}.



\section{Discrete Problem}\label{sec:dgfem}

In this section we discretize the saddle point problem \eqref{subeq:BSPP} by a DG method \cite{arnold2002unified,arnold1982interior,brezzi2004discontinuous}. Let $\mathcal{T}_h$ be a quasi-uniform and shape regular simplicial triangulation of $\Omega$.
The diameter of $T\in\mathcal{T}_h$ is denoted by $h_T$ and $h=\max_{T\in\mathcal{T}_h}h_T$ is the mesh diameter. 
Let $\mathcal{E}_h=\mathcal{E}^b_h\cup\mathcal{E}^i_h$ where $\cE^i_h$ (resp., $\cE^b_h$) represents the set of interior edges (resp., boundary edges).

We further decompose the boundary edges $\cE^b_h$ into the inflow part $\cE^{b,-}_h$ and the outflow part $\cE^{b,+}_h$ which are defined as follows,
\begin{align}
    \cE^{b,-}_h&=\{e\in\cE^b_h: e\subset\{x\in\partial\O: \bz(x)\cdot\bm{n}(x)<0\}\},\\
    \cE^{b,+}_h&=\cE^b_h\setminus\cE^{b,-}_h.
\end{align}
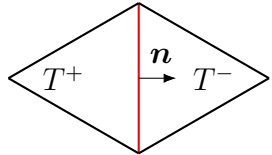
\begin{figure}[h]
\centering
\begin{tikzpicture}
\draw[thick,red] plot coordinates {(0,1) (0,-1)};
\draw[thick] plot coordinates {(-1.732,0) (0,1)};
\draw[thick] plot coordinates {(-1.732,0) (0,-1)};
\draw[thick] plot coordinates {(1.732,0) (0,1)};
\draw[thick] plot coordinates {(1.732,0) (0,-1)};

\node at (-1,0) {$T^+$};
\node at (1,0) {$T^-$};
\draw[-Latex](0,0) -- (0.5,0);
\node at (0.3,0.3) {$\bm{n}$};
\end{tikzpicture}
\caption{Interior edges} \label{figure:ie}
\end{figure}

For an edge $e\in \mathcal{E}^i_h$, let $h_e$ be the length of $e$. For each edge we associate a fixed unit normal $\bm{n}$. We denote by $T^+$ the element for which $\bm{n}$ is the outward normal, and $T^-$ the element for which $-\bm{n}$ is the outward normal (see Figure \ref{figure:ie}). We define the discontinuous finite element space $V_h$ as 
\begin{equation}
    V_h=\{v\in\LT:v|_T\in\mathbb{P}_1(T)\quad\forall\ T\in\mathcal{T}_h\}.
\end{equation}
For $v\in V_h$ on an edge $e$, we define
\begin{equation}
    v^+=v|_{T^+}\quad\text{and}\quad v^-=v|_{T^-}.
\end{equation}
We define the jump and average for $v\in V_h$ on an edge $e$ as follows,
\begin{equation}
    [v]=v^+-v^-,\quad \{v\}=\frac{v^++v^-}{2}.
\end{equation}
For $e\in\mathcal{E}_h^b$ with $e\in\partial T$, we let
\begin{equation}
    [v]=\{v\}=v|_T.
\end{equation}
We also denote 
\begin{equation}
    (w,v)_e:=\int_e wv\ \!ds\quad\text{and}\quad(w,v)_T:=\int_T wv\ \!dx.
\end{equation}

We introduce the following notation (cf. \cite{di2011mathematical}):
\begin{equation}
    \tau_c=\frac{1}{\max(\|\gamma\|_{\infty},\ |\bm{\zeta}|_{1,\infty})}.
\end{equation}
The quantity $\tau_c$ is useful in the convergence analysis. Note that $\tau_c$ is finite because $\|\gamma\|_{\infty}=|\bm{\zeta}|_{1,\infty}=0$ implies $\gamma-\frac12\nabla\cdot\bm{\zeta}=0$, which contradicts our assumption \eqref{eq:advassump}.

\subsection{Discontinuous Galerkin methods}
 DG methods for \eqref{eq:conciseform} aim to find $(\tilde{p}_h,\tilde{y}_h)\in V_h\times V_h$ such that
\begin{equation}\label{eq:dg}
\mathcal{B}_h((\tilde{p}_h,\tilde{y}_h),(q,z))=-(y_d,q)_\LT \quad \quad \forall (q,z)\in V_h\times V_h,
\end{equation}
where
\begin{equation}\label{eq:dgbilinear}
\mathcal{B}_h((p,y),(q,z))=\beta^\frac12a_h(q,p)-(y,q)_\LT+(p,z)_\LT+\beta^\frac12a_h(y,z).
\end{equation}
The bilinear form $a_h(\cdot,\cdot)$ is defined by
\begin{equation}\label{eq:ahdef}
    a_h(u,v)=\eps a_h^{\text{sip}}(u,v)+a^{\text{ar}}_h(u,v)\quad\forall u,v\in V_h,
\end{equation}
where  the  term
\begin{equation}\label{eq:dgbilinearsip}
\begin{aligned}
    a^{\text{sip}}_h(u,v)=&\sum_{T\in\mathcal{T}_h}(\nabla u, \nabla v)_T-\sum_{e\in\mathcal{E}_h}(\{\bm{n}\cdot\nabla u\},[v])_e
    -\sum_{e\in\mathcal{E}_h}(\{\bm{n}\cdot\nabla v\},[u])_e\\
    &+\sigma\sum_{e\in\mathcal{E}_h} h_e^{-1}([u],[v])_e
\end{aligned}
\end{equation}
is the bilinear form of the symmetric interior penalty (SIP) method with sufficiently large penalty parameter $\sigma$.
The upwind DG scheme (cf. \cite{brezzi2004discontinuous,di2011mathematical}) for the advection-reaction term is defined as
\begin{equation}\label{eq:ardef}
         a^{\text{ar}}_h(w,v)=\sum_{T\in\mathcal{T}_h}(\bz\cdot\nabla w+\gamma w, v)_T-\sum_{e\in\cE^i_h}(\bn\cdot\bz[w],v^{\text{down}})_e-\sum_{e\in\cE^{b,-}_h}(\bn\cdot\bz\ \!w,v)_e.
     \end{equation} 
     Here, the downwind value $v^{\text{down}}$ of a function on an interior edge $e\in\cE_h^i$ is defined as
\begin{equation}\label{eq:upwvalue}
    v^{\text{down}}=\left\{
    \begin{aligned}
    v^-\quad\text{if}\quad\bz\cdot\mathbf{n}\ge0,\\
    v^+\quad\text{if}\quad\bz\cdot\mathbf{n}<0.
    \end{aligned}
    \right.
\end{equation}

Note that the scheme \eqref{eq:ardef} is equivalent to the following,
\begin{equation}\label{eq:ardefeq}
\begin{aligned}
    a^{\text{ar}}_h(w,v)&=\sum_{T\in\mathcal{T}_h}(\bz\cdot\nabla w+\gamma w, v)_T-\sum_{e\in\cE^i_h\cup\cE_h^{b,-}}(\bn\cdot\bz[w],\{v\})_e\\
    &\quad+\sum_{e\in\cE^i_h}\frac12 (|\bm{\zeta}\cdot\bn|[w],[v])_e.
\end{aligned}
\end{equation}

DG methods for the more general problem \eqref{eq:regu} aim to find $(p_h,y_h)\in V_h\times V_h$ such that
\begin{equation}\label{eq:dgprob}
\mathcal{B}_h((p_h,y_h),(q,z))=(f,q)_\LT+(g,z)_\LT\quad \quad \forall (q,z)\in V_h\times V_h.
\end{equation}

In this context, we define the norm $\trinorm{\cdot}$ as
\begin{equation}\label{eq:hbhnormt}
  \trinorm{v}^2=\beta^\frac12\|v\|^2_{1,\eps}+\|v\|^2_\LT,
\end{equation}
and the norm $\|\cdot\|_{1,\eps}$ (cf. \cite{di2011mathematical}) as
\begin{equation}\label{eq:hbhnorm}
  \|v\|_{1,\eps}^2:=\|v\|^2_{\Heh}=\eps\|v\|^2_{d}+\|v\|^2_{ar},
\end{equation}
where 
\begin{equation}\label{eq:hbhnorm1}
  \|v\|_{d}^2=\sum_{T\in\mathcal{T}_h}\|\nabla v\|^2_{L_2(T)}+\sum_{e\in\mathcal{E}_h}\frac{1}{h_e}\|[v]\|^2_{L_2(e)}+\sum_{e\in\mathcal{E}_h}h_e\|\{\bm{n}\cdot\nabla v\}\|_{L_2(e)}^2
\end{equation}
and
\begin{equation}\label{eq:hbhnorm2}
  \|v\|_{ar}^2=
  {\tau_c}^{-1}\|v\|^2_{\LT}+\int_{\partial\O}\frac12|\bz\cdot\bn|v^2\ \!ds+\sum_{e\in\cE_h^i}\int_e\frac12|\bz\cdot\bn|[v]^2\ \!ds.
\end{equation}

\subsection{The properties of $a_h(\cdot,\cdot)$}
Let $V=\Ho\cap H^2(\O)$. It is well-known that
        \begin{alignat}{3}
            a^{sip}_h(w,v)&\lesssim\|w\|_{d}\|v\|_{d}&&\quad \forall w,v\in V+V_h,\label{eq:ahsipcont}\\  
            a^{sip}_h(v,v)&\gtrsim\|v\|^2_{d}&&\quad\forall v\in V_h,\label{eq:ahsipcoer}
        \end{alignat}
for sufficiently large $\sigma$ (cf. \cite{di2011mathematical,BS}).
We also have (\cite{di2011mathematical})
\begin{equation}\label{eq:arcoer}
    a^{ar}_h(v,v)\gtrsim\min(1,\gamma_0\tau_c)\|v\|^2_{ar} \quad\forall v\in V_h.
\end{equation}
One can obtain (cf. \cite[Lemma 2.30]{di2011mathematical})
\begin{equation}\label{eq:ahboundstar}
    a_h^{ar}(w-\pi_hw,v)\lesssim\|w-\pi_hw\|_{ar,*}\|v\|_{ar}\quad \forall w\in V, v\in V_h,
\end{equation}
for a stronger norm $\|\cdot\|_{ar,*}$ defined as
\begin{equation}\label{eq:arsnorm}
    \|v\|_{ar,*}^2=\|v\|_{ar}^2+\sum_{T\in\cT_h}\|\bm{\zeta}\|_{0,\infty}\|v\|^2_{L_2(\partial T)}.
    \end{equation}
Here the operator $\pi_h: V\rightarrow V_h$ is the $L_2$ orthogonal projection. Note that the following is also true,
\begin{equation}\label{eq:ahboundstard}
    a_h^{ar}(v,w-\pi_hw)\lesssim\|w-\pi_hw\|_{ar,*}\|v\|_{ar}\quad \forall w\in V, v\in V_h.
\end{equation}
\begin{remark}
    The estimate \eqref{eq:ahboundstard} is not derived from \eqref{eq:ahboundstar} since $a_h^{ar}
    (\cdot,\cdot)$ is nonsymmetric. However, 
    the technique  used in \cite[Lemma 2.30]{di2011mathematical} to derive \eqref{eq:ahboundstar} can be employed to establish \eqref{eq:ahboundstard}.
\end{remark}
Overall, we have
\begin{alignat}{3}
            a_h(w-\pi_hw,v)&\lesssim\|w-\pi_hw\|_{1,\eps,*}\|v\|_{1,\eps}&&\quad \forall w\in V, v\in V_h,\label{eq:ahcontf}\\  
            a_h(v,w-\pi_hw)&\lesssim\|w-\pi_hw\|_{1,\eps,*}\|v\|_{1,\eps}&&\quad \forall w\in V, v\in V_h,\label{eq:ahcontb}\\  
            a_h(v,v)&\gtrsim\min(1,\gamma_0\tau_c)\|v\|^2_{1,\eps}&&\quad\forall v\in V_h,\label{eq:ahcoer}
        \end{alignat}
where the norm $\|\cdot\|_{1,\eps,*}$ is defined as
\begin{equation}\label{eq:hstarnorm}
    \|\cdot\|^2_{1,\eps,*}=\eps\|\cdot\|_d^2+\|\cdot\|_{ar,*}^2.
\end{equation} 

\subsection{The properties of $\cB_h$}

By \eqref{eq:dgbilinear}, \eqref{eq:ahcoer} and a direct calculation, we have
\begin{equation}\label{eq:bhcoer1}
\begin{aligned}
  \mathcal{B}_h&((p,y),(p-y,p+y))\\
  &=\beta^\frac12a_h(p,p)+(p,p)_\LT+(y,y)_\LT+\beta^\frac12a_h(y,y)\\
  &\gtrsim\min(1,\gamma_0\tau_c)\big(\trinorm{p}^2+\trinorm{y}^2\big)
\end{aligned}
\end{equation}
and
\begin{equation}\label{eq:bhcoer2}
  \trinorm{p-y}^2+\trinorm{p+y}^2=2(\trinorm{p}^2+\trinorm{y}^2)
\end{equation}
by the parallelogram law.
It follows from \eqref{eq:bhcoer1} and \eqref{eq:bhcoer2} that
\begin{equation}\label{eq:bhwposed}
\begin{aligned}
  \trinorm{p_h}&+\trinorm{y_h}\\
  &\lesssim \frac{1}{\min(1,\gamma_0\tau_c)} \sup_{(q,z)\in V_h\times V_h}\frac{\mathcal{B}_h((p_h,y_h),(q,z))}{\trinorm{q}+\trinorm{z}}\quad\forall(p_h,y_h)\in V_h\times V_h.
  \end{aligned}
\end{equation}
Define the norm
\begin{equation}\label{eq:tristarnorm}
    \trinorm{v}^2_*=\beta^\frac12\|v\|^2_{1,\eps,*}+\|v\|^2_\LT.
\end{equation}
It follows from \eqref{eq:hbhnormt}, \eqref{eq:ahcontf}, \eqref{eq:ahcontb}, and \eqref{eq:arsnorm} that, for any $(p,y)\in V\times V$ and $(q,z)\in V_h\times V_h$,
\begin{equation}\label{eq:bhcont}
\begin{aligned}
    &\cB_h((p-\pi_h,y-\pi_hy),(q,z))\\
    &=\beta^\frac12a_h(q,p-\pi_hp)-(y-\pi_hy,q)_\LT+(p-\pi_hp,z)_\LT\\
    &\quad+\beta^\frac12a_h(y-\pi_hy,z)\\
    &\lesssim \beta^\frac12\|p-\pi_hp\|_{1,\eps,*}\|q\|_{1,\eps}+\|y-\pi_hy\|_\LT\|q\|_\LT\\
    &\hspace{1cm}+\|p-\pi_hp\|_\LT\|z\|_\LT+\beta^\frac12\|y-\pi_hy\|_{1,\eps,*}\|z\|_{1,\eps}\\
    &\lesssim (\trinorm{p-\pi_hp}_*+\trinorm{y-\pi_hy}_*)(\trinorm{q}+\trinorm{z}).
\end{aligned}
\end{equation}

\subsection{Consistency}

It is well-known that the DG method \eqref{eq:dgprob} is consistent (cf. \cite{arnold2002unified,riviere2008discontinuous,BS,di2011mathematical}). In other words, we have the following Galerkin orthogonality,
\begin{equation}\label{eq:go}
    \cB_h((p-p_h,y-y_h),(q,z))=0 \quad \forall(q,z)\in V_h\times V_h,
\end{equation}
where $(p,y)$ is the solution to \eqref{eq:regu} and $(p_h,y_h)$ is the solution to \eqref{eq:dgprob}.

\section{Convergence Analysis of DG Methods}\label{sec:convanalysis}

In this section, we establish concrete error estimates for the DG method \eqref{eq:dgprob}.
We first recall some preliminary results.
For $T\in\mathcal{T}_h$ and $v\in H^{1+s}(T)$ where $s\in(\frac12,1]$, the following trace inequalities with scaling are standard (cf. \cite[Lemma 7.2]{ern2017finite} and \cite[Proposition 3.1]{ciarlet2013analysis}),
    \begin{align}
         \|v\|_{L_2(\partial T)}&\lesssim(h_T^{-\frac12}\|v\|_{L_2(T)}+h_T^{s-\frac12}|v|_{H^s(T)}),\label{eq:traceinq}\\
         \|\nabla v\|_{L_2(\partial T)}&\lesssim(h_T^{-\frac12}\|\nabla v\|_{L_2(T)}+h_T^{s-\frac12}|\nabla v|_{H^s(T)}).\label{eq:traceinq1}
    \end{align}

We then have the following standard projection estimate \cite{BS}. By \eqref{eq:traceinq}, \eqref{eq:traceinq1} and a standard inverse inequality, we obtain
\begin{equation}\label{eq:projesti}
  \|z-\pi_hz\|_{\LT}+h\|z-\pi_hz\|_d\lesssim h^2\|z\|_{H^2(\O)}\quad\forall z\in V.
\end{equation}
It follows from \eqref{eq:hbhnorm2} that
\begin{equation}\label{eq:projestiar}
    \|z-\pi_hz\|_{ar}\lesssim (\tau_c^{-\frac12}h^2+\|\bm{\zeta}\|_{0,\infty}^\frac12h^{\frac32})\|z\|_{H^2(\O)}\quad\forall z\in V.
\end{equation}

We are ready to state our new error bound.

\begin{theorem}\label{thm:dgesti}
    Let $(p,y)$ be the solution to \eqref{eq:regu} and $(p_h,y_h)$ be the solution to \eqref{eq:dgprob}. We have,
    \begin{equation}\label{eq:dgesti}
    \begin{aligned}
         \trinorm{p-p_h}&+\trinorm{y-y_h}\\
        &\lesssim C_\dagger\Big(\beta^\frac14(\eps^\frac12+\|\bm{\zeta}\|^{\frac12}_{0,\infty}h^\frac12+\tau_c^{-\frac12}h)h+h^2\Big)(\|p\|_{H^2(\O)}+\|y\|_{H^2(\O)}),
    \end{aligned}
    \end{equation}
    where $C_\dagger=(1+\frac{1}{\min(1,\gamma_0\tau_c)})$.
\end{theorem}

\begin{proof}
    It follows from \eqref{eq:bhwposed}, \eqref{eq:bhcont} and \eqref{eq:go} that, for all $(p_h,y_h)\in V_h\times V_h$,
    \begin{equation}\label{eq:err1}
\begin{aligned}
  &\trinorm{p_h-\pi_hp}+\trinorm{y_h-\pi_hy}\\
  &\lesssim \frac{1}{\min(1,\gamma_0\tau_c)} \sup_{(q,z)\in V_h\times V_h}\frac{\mathcal{B}_h((p_h-\pi_hp,y_h-\pi_hy),(q,z))}{\trinorm{q}+\trinorm{z}}\\
  &=\frac{1}{\min(1,\gamma_0\tau_c)}\sup_{(q,z)\in V_h\times V_h}\frac{\mathcal{B}_h((p-\pi_hp,y-\pi_hy),(q,z))}{\trinorm{q}+\trinorm{z}}\\
  &\lesssim \frac{1}{\min(1,\gamma_0\tau_c)}\big( \trinorm{p-\pi_hp}_*+\trinorm{y-\pi_hy}_*\big).
  \end{aligned}
\end{equation}
We then estimate the term $\trinorm{p-\pi_hp}_*$; an estimate 
 of the other term involving $y$ will follow similarly. Combining \eqref{eq:traceinq},  \eqref{eq:projesti}, \eqref{eq:projestiar} and \eqref{eq:hstarnorm}, we obtain,
\begin{equation}\label{eq:epsapprox}
\begin{aligned}
    \trinorm{p-\pi_hp}^2_*&=\beta^\frac12\Big(\eps\|p-\pi_hp\|^2_d+\|p-\pi_hp\|^2_{ar}+\sum_{T\in\cT_h}\|\bm{\zeta}\|_{0,\infty}\|p-\pi_hp\|^2_{L_2(\partial T)}\Big)\\
    &\quad+\|p-\pi_hp\|^2_\LT\\
    &\lesssim \Big(\beta^\frac12(\eps +\|\bm{\zeta}\|_{0,\infty}h+\tau_c^{-1}h^2)h^2+h^4\Big)\|p\|^2_{H^2(\Omega)}.
\end{aligned}
\end{equation}
It follows from \eqref{eq:err1}, \eqref{eq:epsapprox} and the triangle inequality that
\begin{equation}
\begin{aligned}
     \trinorm{p-p_h}&+\trinorm{y-y_h}\\
    &\lesssim C_\dagger\Big(\beta^\frac14(\eps^\frac12+\|\bm{\zeta}\|^{\frac12}_{0,\infty}h^\frac12+\tau_c^{-\frac12}h)h+h^2\Big)(\|p\|_{H^2(\O)}+\|y\|_{H^2(\O)}).
\end{aligned}
\end{equation}
\end{proof}

\begin{remark}\label{remark:dgconv}
     Theorem \ref{thm:dgesti} indicates that our DG methods are optimal in the following sense,
    \begin{equation}
          \|p-p_h\|_{1,\eps}+\|y-y_h\|_{1,\eps}\le\left\{\begin{array}{ll}
              O(\beta^\frac14h+h^2)\quad&\text{if}\ \eqref{eq:stateeq}\ \text{is diffusion-dominated,}\\
              \\
              O(\beta^\frac14h^\frac32+h^2)\quad&\text{if}\ \eqref{eq:stateeq}\ \text{is convection-dominated,}\\
              \\
              O(\beta^\frac14h^2+h^2)\quad&\text{if}\ \eqref{eq:stateeq}\ \text{is reaction-dominated.}
          \end{array}\right.
\end{equation}
Note that $\|p\|_{H^2(\O)}=O(\eps^{-\frac32})$ and $\|y\|_{H^2(\O)}=O(\eps^{-\frac32})$ (\cite{leykekhman2012local}), hence, the estimate \eqref{eq:epsapprox} is not informative when $\eps\le h$. More delicate interior error estimates that stay away from the boundary layers and interior layers for standard DG methods can be found in \cite{leykekhman2012local}.
\end{remark}

\begin{remark}
    The constant $C_\dagger$ in Theorem \ref{thm:dgesti} can be bounded independently of $\gamma$ and $\bm{\zeta}$ due to assumption \eqref{eq:advassump}. The purpose of keeping the constant is to track how the data of the state equation enters the estimate \eqref{eq:dgesti}.
\end{remark}

\section{A Robust Multigrid Preconditioner}\label{sec:mgprecon}

In this section, we discuss block structured multigrid preconditioners to solve the discrete problem \eqref{eq:dgprob}. Our experimental results illustrate their robustness. 
Let the triangulation $\mathcal{T}_1, \mathcal{T}_2, ...$ be generated from the triangulation $\mathcal{T}_0$ through uniform subdivisions such that $h_k\approx\frac12 h_{k-1}$ and $V_k$ be the DG space associated with $\mathcal{T}_k$. Let $\bM_k$ (resp., $\bA_k$) denote the mass matrix representing the bilinear form $(\cdot,\cdot)_\LT$ (resp., $a_h(\cdot,\cdot$)) with respect to the natural discontinuous nodal basis in $V_k$. The discrete problem can be written in the following form,
\begin{equation}\label{eq:mform}
\begin{pmatrix}
    \bM_k & \beta^\frac12\bA_k\\
    \beta^\frac12\bA_k^t & -\bM_k
\end{pmatrix}
\begin{pmatrix}
        \bm{p}\\
        \bm{y}
    \end{pmatrix}=
    \begin{pmatrix}
            \bm{f}\\
            \bm{g}
        \end{pmatrix}.    
\end{equation}
Let $\bm{\mathcal{B}}_k=\begin{pmatrix}
   \bM_k & \beta^\frac12\bA_k\\
    \beta^\frac12\bA_k^t & -\bM_k
\end{pmatrix}$. It has been shown in \cite{pearson2011fast} that the following preconditioner based on the Schur complement is efficient for the problem \eqref{eq:mform},
\begin{equation}\label{eq:preconf}
    \bm{\mathcal{P}}_k=\begin{pmatrix}
        \bM_k & \\
        & \bM_k+\beta\bA_k^t\bM_k^{-1}\bA_k
    \end{pmatrix}.
\end{equation}
In particular, it has been noticed that the eigenvalues of $\bm{\mathcal{P}}_k^{-1}\bm{\mathcal{B}}_k$ are $\{\frac{1-\sqrt{5}}{2},1,\frac{1+\sqrt{5}}{2}\}$. A good approximation of $\bm{\mathcal{P}}_k$ is the following preconditioner (cf. \cite[Theorem 1]{pearson2011fast}),
\begin{equation}\label{eq:precon}
    \widetilde{\bm{\mathcal{P}}}_k=\begin{pmatrix}
        \bM_k & \\
        & (\beta^\frac12\bA_k+\bM_k)^t\bM_k^{-1}(\beta^\frac12\bA_k+\bM_k)
    \end{pmatrix}.
\end{equation}
 First in \cite{pearson2011fast} and  later by other authors in \cite{mardal2022robust}, these preconditioners were used to solve the problem \eqref{optcon}-\eqref{eq:stateeq}. However, they both needed to approximate the mass matrix $\bM_k$ using specific techniques. In our case, the inverse of $\bM_k$ is trivial since the mass matrix for DG methods is block diagonal. For the Schur complement, one has to either efficiently approximate $(\bM_k+\beta\bA^t_h\bM_k^{-1}\bA_k)^{-1}$ or $(\beta^\frac12\bA_k+\bM_k)^{-1}$. The former was accomplished by using isogeometric analysis \cite{mardal2022robust} and the latter can be realized by multigrid \cite{pearson2011fast}. Here we adopt the multigrid strategy proposed in \cite{gopalakrishnan2003multilevel} to efficiently approximate the Schur complement. Note that approximating 
 $(\beta^\frac12\bA_k+\bM_k)^{-1}$ is equivalent to approximately {solving} a single diffusion-convection-reaction equation.

{The quality of the approximate preconditioner can be measured in terms of the distance from the ideal preconditioner, corresponding to the Schur complement $\bS_k=\bM_k+\beta\bA_k^t\bM_k^{-1}\bA_k$. This distance is given by the spectral equivalence between the two matrices; see, e.g., \cite{Faber.Manteuffel.Parter.90}.  In \cite[Theorem 4.1]{pearson2011fast} it was shown
that if 
$$
\widehat \bS_k :=  (\beta^\frac12\bA_k+\bM_k)^t\bM_k^{-1}(\beta^\frac12\bA_k+\bM_k)
$$ 
is used in place of $\bS_k$, 
then the eigenvalues of $\widehat \bS_k^{-1}\bS_k$ are contained in the small interval $[\frac 1 2, 1]$, independently of the problem parameters. This estimate can be approximated in case the exact diagonal block $\widehat \bS_k$ is in turn approximated as $$
\widetilde \bS_k := 
\widetilde \bP_k^t \bM_k^{-1} \widetilde\bP_k, 
$$
where
$\widetilde\bP_k$ is the multigrid operator for $\bP_k=\beta^\frac12\bA_k+\bM_k$. 
Indeed, the eigenvalues of $\widetilde \bS_k^{-1}\bS_k $ can be analyzed by writing the corresponding Rayleigh quotient as follows
$$
\frac{\bv^t \bS_k \bv}{\bv^t \widetilde\bS_k \bv}=
\frac{\bv^t \bS_k \bv}{\bv^t \widehat\bS_k \bv}
\,\,
\frac{\bv^t \widehat\bS_k \bv}{\bv^t \widetilde\bS_k \bv} .
$$
The second factor yields
\begin{eqnarray*}
\frac{\bv^t \widehat\bS_k \bv}{\bv^t \widetilde\bS_k \bv}&=&
\frac{\bv^t \bP_k^t \bM_k^{-1}\bP_k \bv}{\bv^t \widetilde\bP_k^t \bM_k^{-1}\widetilde\bP_k \bv} 
= \frac{\bu^t (\bP_k\widetilde\bP_k^{-1})^t \bM_k^{-1}\bP_k \widetilde\bP_k^{-1}\bu}{\bu^t \bM_k^{-1}\bu},
\end{eqnarray*}
where $\bu =\widetilde\bP_k \bv$, and
$$
\sigma_{\min}(\bP_k\widetilde\bP_k^{-1})^2 
\frac{1}{{\rm cond}(\bM_k)}\le \frac{\bu^t (\bP_k\widetilde\bP_k^{-1})^t \bM_k^{-1}\bP_k \widetilde\bP_k^{-1}\bu}{\bu^t \bM_k^{-1}\bu}
\le
\sigma_{\max}(\bP_k\widetilde\bP_k^{-1})^2 
{{\rm cond}(\bM_k)}.
$$
Here $\sigma_{\min}(\cdot), \sigma_{\max}(\cdot)$ are the minimum and maximum singular values of the argument matrix, and {\rm cond}$(\cdot)$ is the spectral condition number of its argument. We recall that the condition number of $\bM_k$ remains very moderate, independently of the problem parameters.
In summary, we have obtained the following estimates
for the Rayleigh quotient associated with 
$\bS_k \widetilde\bS_k^{-1}$
and any nonzero vector $\bv$,
$$
\frac 1 2 \frac {\sigma_{\min}^2}{{\rm cond}(\bM_k)}
\le \frac{\bv^t \bS_k \bv}{\bv^t \widetilde\bS_k \bv}
\le
 {\sigma_{\max}^2}\cdot {{\rm cond}(\bM_k)};
$$
(a short-hand notation is used for the singular values).
The lower and upper bounds show that the quality of the multigrid operator in approximating the spectral properties of the convection-diffusion operator plays 
 a crucial role for the spectral properties of the whole  preconditioned system.
 Our extensive computational experimentation, some of which is reported below,  seems to show that the designed multigrid operator achieves the goal of making these bounds parameter independent. A rigorous proof remains  an important and challenging open  problem.
}

\subsection{Downwind ordering}\label{sec:dwordering}

It is well-known that reordering the unknowns is crucial for convection-dominated problems. For continuous Galerkin (CG) methods, we refer to \cite{wang1999crosswind,bey1997downwind,hackbusch1997downwind} for more details. For DG methods, it was pointed out in \cite{gopalakrishnan2003multilevel} that downwind ordering of the elements makes the matrix representing the convection term block triangular. We briefly describe an algorithm to order the elements following the convection direction. First, we have the following definitions \cite{lesaint1974finite}.

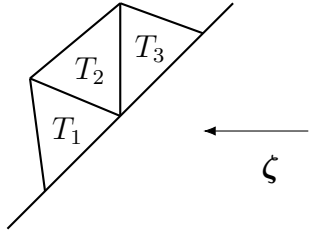
\begin{figure}[h]
\centering
\begin{tikzpicture}
\draw[thick] plot coordinates {(0,0) (3,3)};
\draw[thick] plot coordinates {(0.5,0.5) (0.3,2)};
\draw[thick] plot coordinates {(0.3,2) (1.5,1.5)};
\draw[thick] plot coordinates {(1.5,1.5) (1.5,3)};
\draw[thick] plot coordinates {(1.5,3) (2.6,2.6)};
\draw[thick] plot coordinates {(1.5,3) (0.3,2)};

 \node at (0.8,1.3) {$T_1$};
 \node at (1.1,2.1) {$T_2$};
 \node at (1.9,2.4) {$T_3$};
 \draw[-Latex](4,1.3) -- (2.6,1.3);
 \node at (3.5,0.8) {$\bm{\zeta}$};
\end{tikzpicture}
\caption{Boundary and semi-boundary elements} \label{figure:bdelem}
\end{figure}

\begin{definition}[Boundary elements]
    An element $T\in\cT_h$ is a boundary element if and only if at least one of the edges of $T$ belongs to $\partial \O$.
\end{definition}

\begin{definition}[Semi-boundary elements]
    An element $T\in\cT_h$ is a semi-boundary element if and only if one of the vertices of $T$ belongs to $\partial\O$.
\end{definition}

Now we describe the downwind ordering algorithm as follows in Algorithm \ref{algo:dwordering}.

\begin{algorithm}
\caption{Downwind ordering for DG methods} \label{algo:dwordering}
\begin{algorithmic}[1]
\STATE Find all the boundary elements on the inflow boundary $\cE_h^{b,-}$ and all the semi-boundary elements that have vertices on the inflow boundary $\cE_h^{b,-}$. Denote them as $\{T_i\}_{i=1}^{N}$.
\STATE Reorder all the elements gathered in Step 1 such that the outflow boundary of $T_i$ is the inflow boundary of $T_{j}$ for $i<j$ if $T_i\cap T_j\ne\varnothing$.
\STATE Exclude the elements $\{T_i\}_{i=1}^{N}$ from $\cT_h$ and repeat the process.
\end{algorithmic}
\end{algorithm}


\begin{example}
    In Figure \ref{figure:bdelem}, the elements $T_1$ and $T_3$ are boundary elements while $T_2$ is a semi-boundary element. Since the convection field $\bm{\zeta}$ flows from right to left, the downwind ordering of the elements is $T_3, T_2, T_1$.
\end{example}

\subsection{Multigrid methods for diffusion-convection-reaction equations}\label{sec:mgdcr}

The design of multigrid methods for the diffusion-convection-reaction equations, especially in the convection-dominated regime, is not trivial. Usual components of multigrid would not work well for this problem. This was investigated extensively in \cite{wu2006analysis,olshanskii2004convergence,gopalakrishnan2003multilevel,kim2004uniformly,elman2014finite}. One has to either use a specially designed smoothing step or reorder the unknowns following the flow direction. One key observation is that the smoother used in the multigrid should work for the case when $\eps=0$, i.e, the pure hyperbolic case \cite{olshanskii2004convergence}. Let us consider the state equation \eqref{eq:stateeq} and the corresponding discrete problem at the $k$th level,
\begin{equation}
    \bA_k\bm{w}=\bm{f},
\end{equation}
where $\bA_k$ is the matrix represents $a_h(\cdot,\cdot)$ at the $k$th level.
The following algorithm describes a $V$-cycle algorithm with the forward block Gauss-Seidel smoother $\bm{G}_k$ using downwind ordering in Algorithm \ref{algo:dwordering}. Here $\bm{I}_k^{k-1}$ and $\bm{I}_{k-1}^{k}$ represent standard fine-to-coarse and coarse-to-fine operators respectively. Note that with downwind ordering, the matrix representing the convection term becomes block lower triangular, hence the forward block Gauss-Seidel smoother is efficient \cite{lesaint1974finite,gopalakrishnan2003multilevel}. In the case of linear polynomials, where the diagonal block is $3\times 3$, computing $\bm{G}_k$ is highly efficient.

\begin{algorithm}
\caption{$V$-cycle algorithm for convection-dominated problem with downwind ordering with respect to $\bm{\zeta}$, $MG(k,\bm{f},\bm{u}_0,m_1,m_2)$} \label{algo:dwmg}
\begin{algorithmic}[1]
\STATE Given initial guess $\bm{u}_0$ and $\bm{f}$. 
\STATE If $k=0$, let $\bm{u}=\bA_k^{-1}\bm{f}$, otherwise do the following,
\STATE Pre-smoothing: For $i=1$ to $m_1$,
\STATE \quad $\bm{u}_i=\bm{u}_{i-1}+\bm{G}_k(\bm{f}-\bA_k\bm{u}_{i-1})$.
\STATE Compute $\bm{r}=\bm{I}_k^{k-1}(\bm{f}-\bA_k\bm{u}_{m_1})$. 
\STATE Set $\bm{r}=MG(k-1,\bm{f},\bm{r},m_1,m_2)$.
\STATE Compute $\bm{u}_{m_1+1}=\bm{u}_{m_1}+\bm{I}_{k-1}^k\bm{r}$.
\STATE Post-smoothing: For $i=m_1+2$ to $m_1+m_2+1$,
\STATE \quad $\bm{u}_i=\bm{u}_{i-1}+\bm{G}_k(\bm{f}-\bA_k\bm{u}_{i-1})$.
\end{algorithmic}
\end{algorithm}

For the dual problem \eqref{eq:ospdual}, the downwind ordering with respect to $-\bm{\zeta}$ makes the convection matrix block lower triangular. Hence, the forward block Gauss-Seidel smoother is also efficient for the dual problem.

\begin{remark}
    We do not need to reorder the elements according to $-\bm{\zeta}$ again. Once we have the downwind ordering $\{T_i\}_{i=1}^N$ with respect to $\bm{\zeta}$, the downwind oredering with respect to $-\bm{\zeta}$ is $\{T_i\}_{i=N}^1$. We can then utilize this ordering to solve the dual problem efficiently.
\end{remark}

\subsection{Efficient implementation of the preconditioner \eqref{eq:precon}}

Combining the downwind ordering in Section \ref{sec:dwordering} and the efficient multigrid methods in Section \ref{sec:mgdcr}, we can compute the preconditioner \eqref{eq:precon} efficiently as follows in Algorithm \ref{algo:precon}.

\begin{algorithm}
\caption{Efficient computation of the preconditioner \eqref{eq:precon}} \label{algo:precon}
\begin{algorithmic}[1]
\STATE Given 
$\begin{pmatrix}
    \bm{w}\\
    \bm{v}
\end{pmatrix}$.
\STATE Compute $\bm{w}_1=\bM_k^{-1}\bm{w}$. This step is exact since $\bM_k$ is block-diagonal.
\STATE Compute $\bm{v}_1=MG(k,\bm{v},0,m_1, m_2)$ (apply to $(\beta^\frac12\bA_k+\bM_k)^t$).\label{algo:mgat}
\STATE Compute $\bm{v}_2=\bM_k\bm{v}_1$.
\STATE Compute $\bm{v}_3=MG(k,\bm{v}_2,0,m_1, m_2)$ (apply to $(\beta^\frac12\bA_k+\bM_k)$).\label{algo:mga}
\STATE Output $\widetilde{\bm{\mathcal{P}}}_h^{-1}\begin{pmatrix}
    \bm{w}\\
    \bm{v}
\end{pmatrix}$=$\begin{pmatrix}
    \bm{w}_1\\
    \bm{v}_3
\end{pmatrix}$.
\end{algorithmic}
\end{algorithm}

\begin{remark}\label{remark:tinybgs}
When $\eps$ is tiny, one forward block Gauss-Seidel sweep is enough for Step \ref{algo:mgat} and Step \ref{algo:mga}. Indeed, for the pure hyperbolic case, forward block Gauss-Seidel iteration is an exact solver \cite{gopalakrishnan2003multilevel}. Therefore, for strongly convection-dominated case, Algorithm \ref{algo:precon} is extremely efficient.    
\end{remark}


\section{Numerical Results}\label{sec:numerics}

In this section, we show numerical experiments of the DG methods \eqref{eq:dgprob} and the corresponding preconditioner introduced in the previous section. We solve the discrete problem \eqref{eq:dgprob} using MINRES preconditioned by $\widetilde{\bm{\mathcal{P}}}_k$ defined in \eqref{eq:precon} with tolerance $10^{-6}$.  We use the built-in \verb|minres| function in MATLAB to solve the discrete problem. Steps \ref{algo:mgat} and \ref{algo:mga} in Algorithm \ref{algo:precon} are computed by a single $V$-cycle multigrid method described in Algorithm \ref{algo:dwmg} with $8$ pre-smoothing and post-smoothing steps. To broaden  our comparisons, we  have also used an ILU-preconditioned BiCGSTAB($\ell$) algorithm (default code in Matlab) with tolerance $10^{-8}$ to compute Steps \ref{algo:mgat} and \ref{algo:mga} in Algorithm \ref{algo:precon} as well. In  addition,  for $\eps=10^{-6}$ and $10^{-9}$, we also compute Step \ref{algo:mgat} and Step \ref{algo:mga} in Algorithm \ref{algo:precon} with only one step of backward block Gauss-Seidel iteration and one step of forward block Gauss-Seidel iteration respectively.

We then include the convergence results in the convection-dominated regime to justify our main theorem. We denote $e_y=y-y_h$ and $e_p=p-p_h$ in this section, where $y$, $p$ are solutions to \eqref{eq:regu} and $p_h$, $y_h$ are solutions to the discrete problem \eqref{eq:dgprob}. We compute the global convergence rates of the state and the adjoint state in $L_2$ and $\|\cdot\|_{1,\eps}$ norms. We also compute the local convergence rates of the state and the adjoint state in $L_2$ and $\|\cdot\|_{H^1(\cT_h)}$ norms. Here, the norm $\|\cdot\|_{H^1(\cT_h)}$ is defined as $\|\cdot\|^2_{H^1(\cT_h)}:=\sum_{T\in\cT_h}\|\nabla\cdot\|^2_{H^1(T)}$. We then illustrate the efficiency of our preconditioner by showing the numbers of iteration for the preconditioned MINRES algorithm.

\begin{center}
\begin{table}
\caption{Example \ref{ex:mgcontract}. Contraction numbers of multigrid methods for Step \ref{algo:mga} in Algorithm \ref{algo:dwmg} with $\beta=1$ and different $\eps$} \label{table:contraa}
\begin{tabular}{cccc|ccc}\hline
\multirow{3}{*}{$k$}&\multicolumn{3}{c}{$\eps=10^{-1}$}&\multicolumn{3}{c}{$\eps=10^{-3}$}\\
\cline{2-7}
&\multicolumn{3}{c}{$m$}&\multicolumn{3}{c}{$m$}\\
&$2$&$4$&$8$&$2$&$4$&$8$\\
\hline
$1$\ \vline&9.27e-02&1.55e-02&4.53e-04&1.53e-07&3.12e-14&1.69e-16\\
$2$\ \vline&1.73e-01&4.94e-02&5.84e-03&3.27e-05&6.45e-07&5.16e-16\\
$3$\ \vline&2.55e-01&1.37e-01&4.61e-02&9.28e-04&1.67e-06&1.04e-12\\
$4$\ \vline&3.19e-01&1.77e-01&9.04e-02&1.55e-02&1.72e-04&3.04e-08\\
$5$\ \vline&3.63e-01&2.21e-01&1.19e-01&1.11e-01&8.44e-03&3.88e-05\\
$6$\ \vline&3.81e-01&2.39e-01&1.32e-01&3.17e-01&1.00e-01&1.25e-02\\
$7$\ \vline&3.98e-01&2.43e-01&1.41e-01&3.37e-01&1.72e-01&5.68e-02\\
\hline
\multirow{3}{*}{$k$}&\multicolumn{3}{c}{$\eps=10^{-6}$}&\multicolumn{3}{c}{$\eps=10^{-9}$}\\
\cline{2-7}
&\multicolumn{3}{c}{$m$}&\multicolumn{3}{c}{$m$}\\
&$2$&$4$&$8$&$2$&$4$&$8$\\
\hline
$1$\ \vline&2.44e-16&2.77e-16&2.83e-16&3.41e-16&2.98e-16&1.97e-16\\
$2$\ \vline&4.85e-16&7.84e-16&4.82e-16&3.88e-16&3.60e-16&2.53e-16\\
$3$\ \vline&6.68e-15&5.38e-16&5.66e-16&6.38e-16&5.34e-16&5.32e-16\\
$4$\ \vline&1.68e-12&9.73e-16&1.05e-15&9.57e-16&1.24e-15&1.15e-15\\
$5$\ \vline&3.80e-06&1.46e-15&1.73e-15&1.35e-15&1.63e-15&1.54e-15\\
$6$\ \vline&3.18e-08&1.83e-15&1.86e-15&1.95e-15&2.16e-15&2.23e-15\\
$7$\ \vline&2.74e-06&1.90e-11&2.69e-15&2.94e-15&2.89e-15&2.84e-15\\
\hline
\end{tabular}
\end{table}
\end{center}

 \begin{example}[Multigrid Methods for Convection-dominated Problems]\label{ex:mgcontract}
     In this example, we first illustrate the contraction behaviors of the multigrid methods described in Algorithm \ref{algo:dwmg}. Note that Algorithm \ref{algo:dwmg} is a crucial component of the preconditioner described in Algorithm \ref{algo:precon}. We compute the contraction numbers of the multigrid methods for both the forward problem (Step \ref{algo:mga} in Algorithm \ref{algo:precon}) and the dual problem (Step \ref{algo:mgat} in Algorithm \ref{algo:precon}) with $m$ smoothing steps.
 \end{example}

We first consider the case with different values of $\eps$ where $\beta=1$. As one can see from Tables \ref{table:contraa} and \ref{table:contraat}, our multigrid methods are highly efficient in convection-dominated regime, especially in the cases where $\eps=10^{-6}$ and $\eps=10^{-9}$. Indeed, as pointed out in \cite{gopalakrishnan2003multilevel}, with downwind ordering, the block Gauss-Seidel iteration itself is almost a direct solver in these cases. For mild convection-dominated cases, where $\eps=10^{-3}$, one can see our multigrid methods also perform well. For the case $\eps=10^{-1}$, the convergence behavior of the multigrid methods tends to the classical $O(m^{-1})$ convergence rate as in the diffusion-dominated case. 
Overall, this example shows that, with downwind ordering, the multigrid methods with a block Gauss-Seidel smoother are extremely suitable for convection-dominated problems.

We then report the contraction numbers in Table \ref{table:contraabeta} with different values of $\beta$ where $\eps=10^{-3}$. For simplicity, we only include the results at higher levels for Step \ref{algo:mga} in Algorithm \ref{algo:precon}. One can clearly see that the contraction numbers for Algorithm \ref{algo:dwmg} are small for all $\beta$ values, and they decrease when $\beta$ decreases. This is because the block Gauss-Seidel algorithm tends to an exact solver with any ordering as $\beta\rightarrow0$, due to the fact that $\bm{M}_k$ is block-diagonal.

\begin{center}
\begin{table}
\caption{Example \ref{ex:mgcontract}. Contraction numbers of multigrid methods for Step \ref{algo:mgat} in Algorithm \ref{algo:dwmg} with $\beta=1$ and different $\eps$}\label{table:contraat}
\begin{tabular}{cccc|ccc}\hline
\multirow{3}{*}{$k$}&\multicolumn{3}{c}{$\eps=10^{-1}$}&\multicolumn{3}{c}{$\eps=10^{-3}$}\\
\cline{2-7}
&\multicolumn{3}{c}{$m$}&\multicolumn{3}{c}{$m$}\\
&$2$&$4$&$8$&$2$&$4$&$8$\\
\hline
$1$\ \vline&9.27e-02&1.55e-02&4.53e-04&1.08e-07&3.49e-14&3.59e-16\\
$2$\ \vline&1.75e-01&4.18e-02&7.65e-03&9.37e-06&5.39e-07&4.45e-16\\
$3$\ \vline&2.48e-01&1.19e-01&4.51e-02&1.21e-03&1.05e-06&1.09e-11\\
$4$\ \vline&3.25e-01&1.80e-01&8.02e-02&1.70e-02&1.85e-04&2.97e-08\\
$5$\ \vline&3.68e-01&2.28e-01&1.16e-01&9.06e-02&6.12e-03&4.42e-05\\
$6$\ \vline&3.82e-01&2.41e-01&1.38e-01&2.44e-01&6.08e-02&8.76e-03\\
$7$\ \vline&3.99e-01&2.47e-01&1.40e-01&3.35e-01&1.82e-01&6.35e-02\\
\hline
\multirow{3}{*}{$k$}&\multicolumn{3}{c}{$\eps=10^{-6}$}&\multicolumn{3}{c}{$\eps=10^{-9}$}\\
\cline{2-7}
&\multicolumn{3}{c}{$m$}&\multicolumn{3}{c}{$m$}\\
&$2$&$4$&$8$&$2$&$4$&$8$\\
\hline
$1$\ \vline&2.06e-16&2.13e-16&2.42e-16&2.62e-16&2.45e-16&4.18e-16\\
$2$\ \vline&5.21e-16&3.77e-16&2.70e-16&5.89e-16&4.39e-16&5.41e-16\\
$3$\ \vline&9.52e-15&8.57e-16&6.50e-16&6.54e-16&8.76e-16&9.90e-16\\
$4$\ \vline&1.93e-12&9.85e-16&1.20e-15&9.74e-16&1.05e-15&1.12e-15\\
$5$\ \vline&4.18e-06&1.35e-15&1.31e-15&1.55e-15&1.41e-15&1.46e-15\\
$6$\ \vline&2.90e-08&2.11e-15&1.85e-15&2.17e-15&2.05e-15&1.99e-15\\
$7$\ \vline&2.72e-06&2.08e-11&2.57e-15&3.29e-15&2.83e-15&3.28e-15\\
\hline
\end{tabular}
\end{table}
\end{center}

\begin{center}
\begin{table}
\caption{Example \ref{ex:mgcontract}. Contraction numbers of multigrid methods for Step \ref{algo:mga} in Algorithm \ref{algo:dwmg} with $\eps=10^{-3}$ and different values of $\beta$}\label{table:contraabeta}
\begin{tabular}{cccc|ccc}\hline
\multirow{3}{*}{$k$}&\multicolumn{3}{c}{$\beta=10^{-1}$}&\multicolumn{3}{c}{$\beta=10^{-2}$}\\
\cline{2-7}
&\multicolumn{3}{c}{$m$}&\multicolumn{3}{c}{$m$}\\
&$2$&$4$&$8$&$2$&$4$&$8$\\
\hline
$5$\ \vline&1.07e-01&7.61e-03&2.79e-05&6.44e-02&4.00e-03&2.01e-05\\
$6$\ \vline&2.61e-01&1.10e-01&9.25e-03&1.80e-01&5.80e-02&7.75e-03\\
$7$\ \vline&3.22e-01&1.72e-01&5.75e-02&2.81e-01&1.40e-01&3.57e-02\\
\hline
\multirow{3}{*}{$k$}&\multicolumn{3}{c}{$\beta=10^{-4}$}&\multicolumn{3}{c}{$\beta=10^{-8}$}\\
\cline{2-7}
&\multicolumn{3}{c}{$m$}&\multicolumn{3}{c}{$m$}\\
&$2$&$4$&$8$&$2$&$4$&$8$\\
\hline
$5$\ \vline&5.84e-04&2.73e-07&1.42e-14&4.50e-11&4.98e-16&4.50e-16\\
$6$\ \vline&2.08e-02&4.02e-04&1.16e-07&4.91e-08&5.32e-16&4.59e-16\\
$7$\ \vline&9.50e-02&1.50e-02&5.73e-04&3.98e-07&7.27e-13&4.56e-16\\
\hline
\end{tabular}
\end{table}
\end{center}

\begin{example}[Smooth Solutions]\label{ex:smsol}
    In this example, we take $\Omega=[0,1]^2$, $\gamma=0$, $\bm{\zeta}=[1,0]^t$ and let the exact solutions of \eqref{eq:regu} be 
    \begin{equation}
        y=x_1(1-x_1)x_2(1-x_2)\quad\text{and}\quad p=\sin(2\pi x_1)\sin(2\pi x_2).
    \end{equation}
    We take $\beta=1$ unless otherwise stated.
\end{example}

We first report the global convergence results of the methods \eqref{eq:dgprob} with $\eps=10^{-9}$ in Table \ref{table:smsolg}. We observe $O(h^2)$ convergence for $\|e_y\|_\LT$ and $\|e_p\|_\LT$. They are better than the theoretical results in Theorem \ref{thm:dgesti}, which is due to the smoothness of the solutions. Similar convergence behaviors were also observed in \cite{ayuso2009discontinuous}. We also observe almost $O(h^2)$ convergence for $\|e_y\|_{1,\eps}$ and $O(h^\frac32)$ convergence for $\|e_p\|_{1,\eps}$. Again, due to the smoothness of the solutions, we see higher convergence rates in $\|e_y\|_{1,\eps}$. We also test and report the local convergence results with $\eps=10^{-9}$ in Table \ref{table:smsoll}. Here we measure the $L_2$ and $\|\cdot\|_{H^1(\cT_h)}$ errors in the domain $[0.25, 0.75]^2$. One can clearly see optimal convergence rates in $L_2$ and $\|\cdot\|_{H^1(\cT_h)}$ norms for both variables. This is consistent with the results in \cite{leykekhman2012local}.

We then show the MINRES numbers of iterations in Table \ref{table:smminres1} for various $\eps$ and different implementations of the preconditioner. We clearly see that the preconditioner \eqref{eq:precon} is robust with respect to $\eps$. Moreover, the performance of the multigrid implementation of the preconditioner matches with the behavior of the contraction numbers in Example \ref{ex:mgcontract}. Indeed, for $\eps=10^{-6}$ and $\eps=10^{-9}$, the multigrid method is almost an exact solver, hence the MINRES numbers of iterations are identical to those of BiCGSTAB. For $\eps=10^{-3}$ and $\eps=10^{-1}$, the MINRES numbers of iterations are still bounded with respect to $k$ which is consistent with the results in Example~\ref{ex:mgcontract}. We also see that for $\eps=10^{-6}$ and $\eps=10^{-9}$, one sweep of backward and forward block Gauss-Seidel is enough (see Remark \ref{remark:tinybgs}).

Lastly, we report the MINRES numbers of iterations in Table \ref{table:smminresb} for various $\beta$ and different implementations of the preconditioner. We take $\eps=10^{-3}$ in Table \ref{table:smminresb}. One can see that the preconditioner is robust with respect to $\beta$ as well.

\begin{table}[htb]
\caption{Convergence rates for Example \ref{ex:smsol} with $\varepsilon=10^{-9}$ (Global)}\label{table:smsolg}
\begin{tabular}{ccccccccc}\hline
\\
$k$&$\|e_y\|_\LT$&Order&$\|e_y\|_{1,\eps}$&Order&$\|e_p\|_\LT$&Order&$\|e_p\|_{1,\eps}$&Order\\
\\
\hline
$1$&1.11e-02&-&1.51e-02&-&5.25e-03&-&5.36e-03&-\\
$2$&5.55e-02&-2.32&5.76e-02&-1.93&1.42e-01&-4.76&2.51e-01&-5.55\\
$3$&1.48e-02&1.90&1.52e-02&1.92&3.58e-02&1.99&9.97e-02&1.33\\
$4$&3.78e-03&1.97&3.87e-03&1.98&8.92e-03&2.00&3.63e-02&1.46\\
$5$&9.50e-04&1.99&9.81e-04&1.98&2.23e-03&2.00&1.29e-02&1.49\\
$6$&2.38e-04&2.00&2.52e-04&1.96&5.57e-04&2.00&4.56e-03&1.50\\
$7$&5.94e-05&2.00&6.60e-05&1.93&1.39e-04&2.00&1.61e-03&1.50\\
$8$&1.49e-05&2.00&1.80e-05&1.88&3.48e-05&2.00&5.68e-04&1.50\\
\hline
\end{tabular}
\end{table}

\begin{table}[htb]
\caption{Convergence rates for Example \ref{ex:smsol} with $\varepsilon=10^{-9}$ (Local)}\label{table:smsoll}
\begin{tabular}{ccccccccc}\hline
\\
$k$&$\|e_y\|_\LT$&Order&$\|e_y\|_{H^1(\cT_h)}$&Order&$\|e_p\|_\LT$&Order&$\|e_p\|_{H^1(\cT_h)}$&Order\\
\\
\hline
$1$&2.95e-03&-&1.48e-02&-&1.64e-03&-&8.88e-03&-\\
$2$&1.01e-02&-1.77&1.13e-01&-2.94&3.60e-02&-4.45&4.60e-01&-5.69\\
$3$&5.16e-03&0.96&4.71e-02&1.26&1.51e-02&1.25&3.76e-01&0.29\\
$4$&1.52e-03&1.77&1.80e-02&1.39&4.09e-03&1.89&2.05e-01&0.88\\
$5$&4.01e-04&1.92&7.44e-03&1.27&1.05e-03&1.96&1.06e-01&0.96\\
$6$&1.03e-04&1.96&3.31e-03&1.17&2.67e-04&1.98&5.35e-02&0.98\\
$7$&2.60e-05&1.98&1.55e-03&1.10&6.72e-05&1.99&2.69e-02&0.99\\
$8$&6.54e-06&1.99&7.47e-04&1.05&1.69e-05&1.99&1.35e-02&0.99\\
\hline
\end{tabular}
\end{table}

\begin{table}[htb]
\caption{MINRES numbers of iterations for Example \ref{ex:smsol}}\label{table:smminres1}
\begin{tabular}{cllll|ll|llll}
&\multicolumn{4}{c}{MG}&\multicolumn{2}{c}{BGS}&\multicolumn{4}{c}{BiCGSTAB}\\
\hline
\multirow{2}{*}{$k$}&\multicolumn{10}{c}{$\eps$}\\
\cline{2-11}\\[-2ex]
&$10^{-1}$&$10^{-3}$&$10^{-6}$&$10^{-9}$&$10^{-6}$&$10^{-9}$&$10^{-1}$&$10^{-3}$&$10^{-6}$&$10^{-9}$\\
\hline
$1$\ \vline&12&15&14&14&14&14&12&15&14&14\\
$2$\ \vline&14&15&14&14&14&14&14&15&14&14\\
$3$\ \vline&16&17&14&14&14&14&14&17&14&14\\
$4$\ \vline&18&17&13&13&13&13&14&17&13&13\\
$5$\ \vline&19&17&13&11&13&11&14&17&13&11\\
$6$\ \vline&20&17&13&11&15&11&14&17&13&11\\
$7$\ \vline&22&19&15&11&20&11&15&17&15&11\\
\hline
\end{tabular}
\end{table}

\begin{table}[htb]
\caption{MINRES numbers of iterations for Example \ref{ex:smsol} with different values of $\beta$ and $\eps=10^{-3}$}\label{table:smminresb}
\begin{tabular}{cllll|llll}
&\multicolumn{4}{c}{MG}&\multicolumn{4}{c}{BiCGSTAB}\\
\hline
\multirow{2}{*}{$k$}&\multicolumn{8}{c}{$\beta$}\\
\cline{2-9}\\[-2ex]
&$10^{-1}$&$10^{-2}$&$10^{-4}$&$10^{-8}$&$10^{-1}$&$10^{-2}$&$10^{-4}$&$10^{-8}$\\
\hline
$1$\ \vline&17&19&8&3&17&19&8&3\\
$2$\ \vline&17&22&12&4&17&22&12&4\\
$3$\ \vline&18&20&14&4&18&20&14&4\\
$4$\ \vline&17&20&17&4&17&20&17&4\\
$5$\ \vline&17&19&18&4&17&19&18&4\\
$6$\ \vline&16&19&19&4&16&19&19&4\\
$7$\ \vline&17&21&19&5&16&20&19&5\\
\hline
\end{tabular}
\end{table}

\begin{example}[Boundary Layer]\label{ex:bdlayer}
     In this example, we take $\Omega=[0,1]^2$, $\beta=1$, $\gamma=0$, $\bm{\zeta}=[\sqrt{2}/2,\sqrt{2}/2]^t$ and let the exact solutions of \eqref{eq:regu} be $y=\eta(x)\eta(y)$ and $p=\eta(1-x)\eta(1-y)$, where
\begin{equation}
  \eta(z)=z^3-\frac{e^{\frac{z-1}{\varepsilon}}-e^{-1/\varepsilon}}{1-e^{-1/\varepsilon}}.
\end{equation}
It is known \cite{leykekhman2012local} that the solution $y$ has a boundary layer near $x=1$ and $y=1$ and solution $p$ has a boundary layer near $x=0$ and $y=0$, when $\eps$ goes to $0$.
\end{example}

\begin{figure}[ht]
    \centering
    \subfloat[Numerical solution $y_h$]{\includegraphics[height=2in]{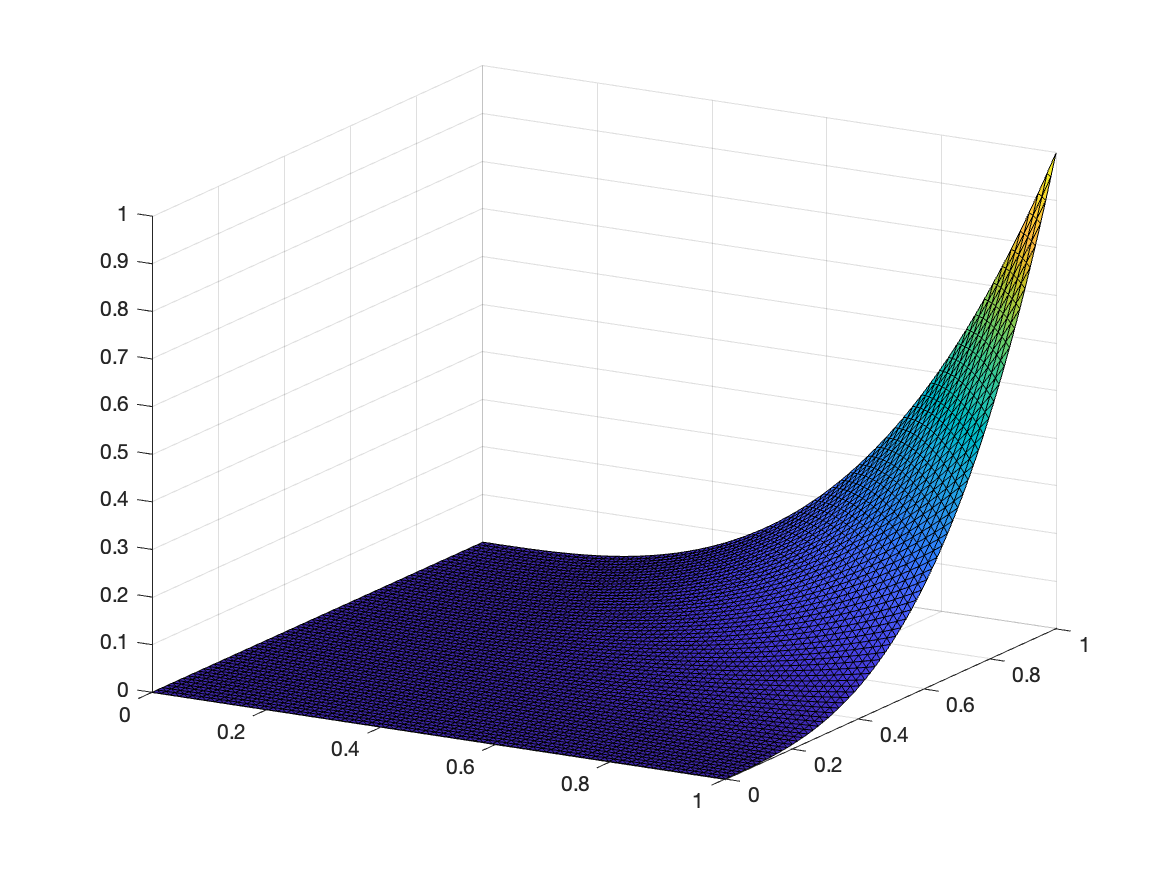}} 
    \subfloat[Numerical solution $p_h$]{\includegraphics[height=2in]{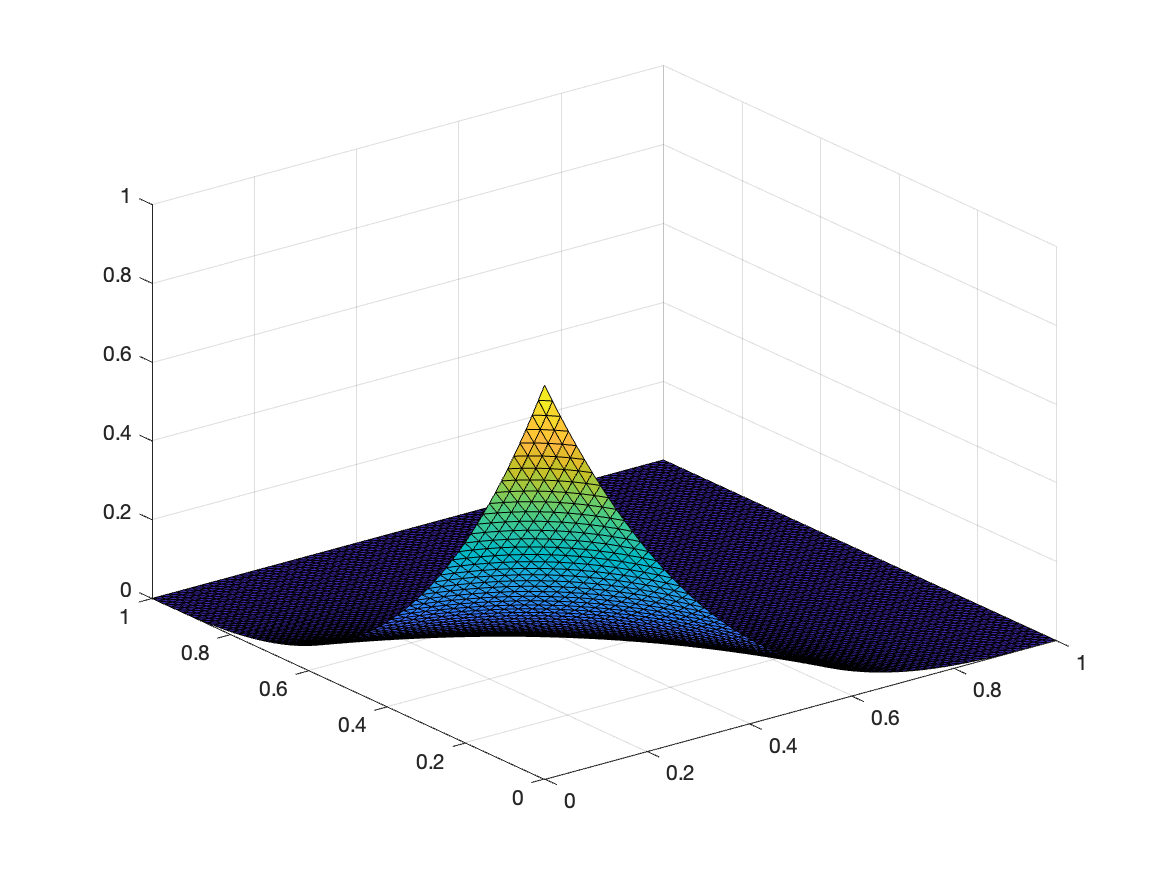}}
    \vfill
    \subfloat[Exact solution $y$]{\includegraphics[height=2in]{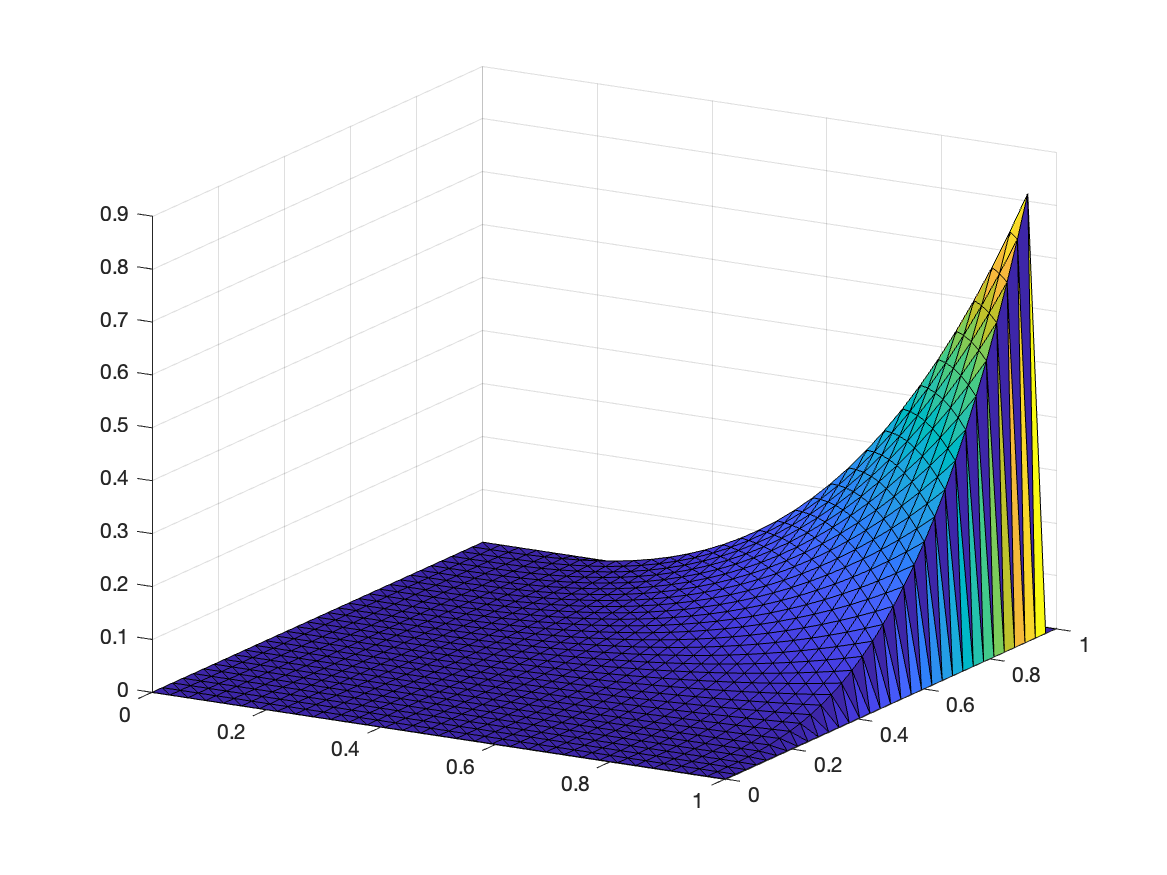}} 
    \subfloat[Exact solution $p$]{\includegraphics[height=2in]{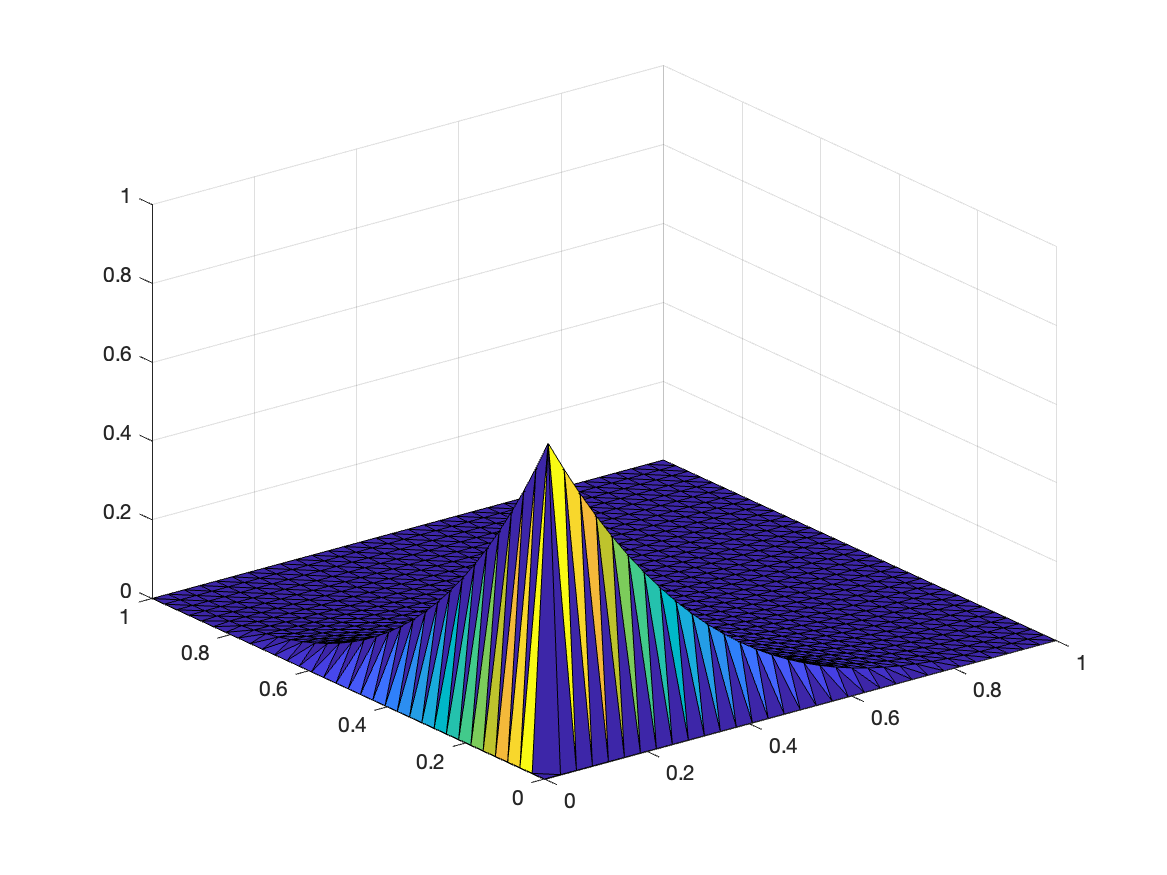}}
    \caption{Numerical solutions and exact solutions}\label{fig:bdlayernumer}
\end{figure}

We first show the global convergence results of the methods \eqref{eq:dgprob} with $\eps=10^{-9}$ for Example~\ref{ex:bdlayer}. We can see from Table \ref{table:bdlayerg} that the global convergence of the state and the adjoint state is $O(h^\frac12)$ in $L_2$ and $\|\cdot\|_{1,\eps}$ norms. 
These deteriorated convergence rates are caused by the sharp boundary layers presented near the outflow boundary. See Figure \ref{fig:bdlayernumer} for the comparison between numerical solutions and exact solutions. One can easily see that the boundary layers are ignored due to the weak treatment of the boundary conditions.

On the other hand, we measure the errors in the interior of the domain $[0.25,0.75]^2$, which is away from the boundary layers. We found that the convergence rates are optimal in $L_2$ and $\|\cdot\|_{H^1(\cT_h)}$ norms, as can be seen from Table \ref{table:bdlayerl}. This illustrates the advantages of DG methods for optimal control problems, as the boundary layers do not pollute the solutions into the interior, where the solution is smooth (cf. \cite{leykekhman2012local}). Again, this is due to the fact that DG methods impose the boundary conditions weakly. This is in contrast to methods that impose the boundary conditions strongly, for example, the SUPG method \cite{heinkenschloss2010local},  in which the oscillations propagate into the interior and one can at most expect $O(h)$ convergence for any polynomial degrees.

We then show the MINRES numbers of iterations in Table \ref{table:bdminres} for various $\eps$ and different implementations of the preconditioner. We again observe that the preconditioner \eqref{eq:precon} is robust with respect to $\eps$. Similar MINRES numbers of iterations are observed for the multigrid preconditioner, as well as the block Gauss-Seidel iterations for small values of $\eps$. We also report the MINRES numbers of iterations in Table \ref{table:bdminres1} for different values of $\beta$. We again observe the robustness of the preconditioner with respect to $\beta$.


\begin{table}[htb]
\caption{Convergence rates for Example \ref{ex:bdlayer} with $\varepsilon=10^{-9}$ (Global)}\label{table:bdlayerg}
\begin{tabular}{ccccccccc}\hline
\\
$k$&$\|e_y\|_\LT$&Order&$\|e_y\|_{1,\eps}$&Order&$\|e_p\|_\LT$&Order&$\|e_p\|_{1,\eps}$&Order\\
\\
\hline
$1$&1.33e-01&-&1.35e-01&-&1.35e-01&-&1.37e-01&-\\
$2$&1.28e-01&0.06&1.29e-01&0.07&1.29e-01&0.07&1.30e-01&0.08\\
$3$&1.03e-01&0.32&1.03e-01&0.32&1.03e-01&0.32&1.03e-01&0.33\\
$4$&7.59e-02&0.44&7.59e-02&0.44&7.59e-02&0.44&7.59e-02&0.44\\
$5$&5.43e-02&0.48&5.43e-02&0.48&5.43e-02&0.48&5.43e-02&0.48\\
$6$&3.85e-02&0.50&3.85e-02&0.50&3.85e-02&0.50&3.85e-02&0.50\\
$7$&2.73e-02&0.50&2.73e-02&0.50&2.73e-02&0.50&2.73e-02&0.50\\
$8$&1.93e-02&0.50&1.93e-02&0.50&1.93e-02&0.50&1.93e-02&0.50\\
\hline
\end{tabular}
\end{table}


\begin{table}[htb]
\caption{Convergence rates for Example \ref{ex:bdlayer} with $\varepsilon=10^{-9}$ (Local)}\label{table:bdlayerl}
\begin{tabular}{ccccccccc}\hline
\\
$k$&$\|e_y\|_\LT$&Order&$\|e_y\|_{H^1(\cT_h)}$&Order&$\|e_p\|_\LT$&Order&$\|e_p\|_{H^1(\cT_h)}$&Order\\
\\
\hline
$1$&3.86e-03&-&2.16e-02&-&2.58e-03&-&1.57e-02&-\\
$2$&5.88e-04&2.72&7.19e-03&1.59&3.08e-04&3.07&4.25e-03&1.88\\
$3$&3.67e-04&0.68&6.06e-03&0.25&1.86e-04&0.72&6.10e-03&-0.52\\
$4$&1.32e-04&1.48&4.18e-03&0.54&6.97e-05&1.42&4.40e-03&0.47\\
$5$&3.88e-05&1.76&2.46e-03&0.77&2.09e-05&1.74&2.60e-03&0.76\\
$6$&1.05e-05&1.88&1.33e-03&0.88&5.72e-06&1.87&1.41e-03&0.89\\
$7$&2.74e-06&1.94&6.93e-04&0.94&1.49e-06&1.94&7.31e-04&0.94\\
$8$&6.98e-07&1.97&3.53e-04&0.97&3.81e-07&1.97&3.73e-04&0.97\\
\hline
\end{tabular}
\end{table}

\begin{table}[htb]
\caption{MINRES number of iterations for Example \ref{ex:bdlayer}}\label{table:bdminres}
\begin{tabular}{cllll|ll|llll}
&\multicolumn{4}{c}{MG}&\multicolumn{2}{c}{BGS}&\multicolumn{4}{c}{BiCGSTAB}\\
\hline
\multirow{2}{*}{$k$}&\multicolumn{10}{c}{$\eps$}\\
\cline{2-11}\\[-2ex]
&$10^{-1}$&$10^{-3}$&$10^{-6}$&$10^{-9}$&$10^{-6}$&$10^{-9}$&$10^{-1}$&$10^{-3}$&$10^{-6}$&$10^{-9}$\\
\hline
$1$\ \vline&12&13&13&13&13&13&12&13&13&13\\
$2$\ \vline&13&15&16&16&16&16&13&15&16&16\\
$3$\ \vline&16&17&17&17&17&17&14&17&17&17\\
$4$\ \vline&18&17&19&19&19&19&15&17&19&19\\
$5$\ \vline&20&19&19&19&19&19&16&19&19&19\\
$6$\ \vline&22&22&20&20&20&20&16&20&20&20\\
$7$\ \vline&24&23&20&20&26&20&17&20&20&20\\
\hline
\end{tabular}
\end{table}

\begin{table}[htb]
\caption{MINRES numbers of iterations for Example \ref{ex:bdlayer} with different values of $\beta$ and $\eps=10^{-3}$}\label{table:bdminres1}
\begin{tabular}{cllll|llll}
&\multicolumn{4}{c}{MG}&\multicolumn{4}{c}{BiCGSTAB}\\
\hline
\multirow{2}{*}{$k$}&\multicolumn{8}{c}{$\beta$}\\
\cline{2-9}\\[-2ex]
&$10^{-1}$&$10^{-2}$&$10^{-4}$&$10^{-8}$&$10^{-1}$&$10^{-2}$&$10^{-4}$&$10^{-8}$\\
\hline
$1$\ \vline&15&18&11&5&15&18&11&5\\
$2$\ \vline&18&22&14&5&18&22&14&5\\
$3$\ \vline&20&23&17&5&20&23&17&5\\
$4$\ \vline&21&23&21&5&21&23&21&5\\
$5$\ \vline&21&25&23&6&21&25&23&6\\
$6$\ \vline&24&25&24&8&22&25&24&8\\
$7$\ \vline&27&27&24&11&22&25&24&11\\
\hline
\end{tabular}
\end{table}

\begin{example}[Interior Layer]\label{ex:inlayer}
    In this example, we take $\Omega=[0,1]^2$, $\gamma=0$, $\bm{\zeta}=[1,0]^t$ and let the exact solutions of \eqref{eq:regu} be 
    \begin{equation}
        y=(1-x_1)^3\arctan(\frac{x_2-0.5}{\eps})\quad\text{and}\quad p=x_1(1-x_1)x_2(1-x_2).
    \end{equation}
    The exact state $y$ has an interior layer along the line $x_2=0.5$ for small $\eps$. We take $\beta=1$ unless otherwise stated.
\end{example}

We show the global convergence results for $\eps=10^{-9}$ in Table \ref{table:inlayerg}. We see that the convergence rates in $L_2$ norm for the state and the adjoint state are $O(h^\frac32)$, which coincide with Theorem \ref{thm:dgesti}. We also observe $O(h)$ convergence for the state in $\|\cdot\|_{1,\eps}$ norm, which is caused by the interior layer. The convergence rate of the adjoint state in $\|\cdot\|_{1,\eps}$ norm is $O(h^\frac32)$ which is optimal in the sense of Remark \ref{remark:dgconv}. The local convergence results in Table \ref{table:inlayerl} are measured in the domain $[0.6,1]\times[0,1]$. The rates are all optimal in $L_2$ and $\|\cdot\|_{H^1(\cT_h)}$ norms, which again, shows that the interior layer does not pollute the solutions into the domain where the solutions are smooth.

We then show the MINRES numbers of iterations in Tables \ref{table:inlayerminres} and \ref{table:inlayerminres1} for various values of $\eps$ and $\beta$ respectively as well as for different implementations of the preconditioner. Similar results are observed as those of previous examples.

\begin{table}[htb]
\caption{Convergence rates for Example \ref{ex:inlayer} with $\varepsilon=10^{-9}$ (Global)}\label{table:inlayerg}
\begin{tabular}{ccccccccc}\hline
\\
$k$&$\|e_y\|_\LT$&Order&$\|e_y\|_{1,\eps}$&Order&$\|e_p\|_\LT$&Order&$\|e_p\|_{1,\eps}$&Order\\
\\
\hline
$1$&1.92e-01&-&2.14e-01&-&9.13e-02&-&9.41e-02&-\\
$2$&6.80e-02&1.50&9.03e-02&1.24&3.18e-02&1.52&3.25e-02&1.54\\
$3$&2.27e-02&1.58&3.87e-02&1.22&1.09e-02&1.55&1.11e-02&1.55\\
$4$&7.67e-03&1.57&1.75e-02&1.14&3.74e-03&1.54&3.80e-03&1.54\\
$5$&2.63e-03&1.54&8.27e-03&1.08&1.30e-03&1.53&1.32e-03&1.53\\
$6$&9.16e-04&1.52&4.01e-03&1.04&4.54e-04&1.51&4.61e-04&1.51\\
$7$&3.21e-04&1.51&1.97e-03&1.02&1.60e-04&1.51&1.62e-04&1.51\\
$8$&1.13e-04&1.51&9.78e-04&1.01&5.64e-05&1.50&5.73e-05&1.50\\
\hline
\end{tabular}
\end{table}

\begin{table}[htb]
\caption{Convergence rates for Example \ref{ex:inlayer} with $\varepsilon=10^{-9}$ (Local)}\label{table:inlayerl}
\begin{tabular}{ccccccccc}\hline
\\
$k$&$\|e_y\|_\LT$&Order&$\|e_y\|_{H^1(\cT_h)}$&Order&$\|e_p\|_\LT$&Order&$\|e_p\|_{H^1(\cT_h)}$&Order\\
\\
\hline
$1$&1.46e-01&-&5.40e-01&-&7.54e-02&-&3.02e-01&-\\
$2$&4.62e-02&1.65&3.48e-01&0.64&2.14e-02&1.82&1.44e-01&1.07\\
$3$&5.85e-03&2.98&7.37e-02&2.24&2.42e-03&3.14&1.84e-02&2.97\\
$4$&1.53e-03&1.93&3.69e-02&1.00&6.41e-04&1.92&7.38e-03&1.32\\
$5$&3.91e-04&1.97&1.84e-02&1.00&1.65e-04&1.96&3.15e-03&1.23\\
$6$&9.67e-05&2.02&9.02e-03&1.03&4.08e-05&2.01&1.40e-03&1.17\\
$7$&2.40e-05&2.01&4.46e-03&1.02&1.02e-05&2.01&6.53e-04&1.10\\
$8$&6.02e-06&2.00&2.23e-03&1.00&2.55e-06&1.99&3.16e-04&1.05\\
\hline
\end{tabular}
\end{table}

\begin{table}[htb]
\caption{MINRES number of iterations for Example \ref{ex:inlayer}}\label{table:inlayerminres}
\begin{tabular}{cllll|ll|llll}
&\multicolumn{4}{c}{MG}&\multicolumn{2}{c}{BGS}&\multicolumn{4}{c}{BiCGSTAB}\\
\hline
\multirow{2}{*}{$k$}&\multicolumn{10}{c}{$\eps$}\\
\cline{2-11}\\[-2ex]
&$10^{-1}$&$10^{-3}$&$10^{-6}$&$10^{-9}$&$10^{-6}$&$10^{-9}$&$10^{-1}$&$10^{-3}$&$10^{-6}$&$10^{-9}$\\
\hline
$1$\ \vline&12&15&14&14&14&14&12&15&14&14\\
$2$\ \vline&14&15&14&14&14&14&14&15&14&14\\
$3$\ \vline&16&15&13&13&13&13&14&15&13&13\\
$4$\ \vline&18&15&13&13&13&13&14&15&13&13\\
$5$\ \vline&18&15&11&11&11&11&14&15&11&11\\
$6$\ \vline&20&17&11&11&13&11&14&15&11&11\\
$7$\ \vline&22&19&13&11&17&11&14&15&13&11\\
\hline
\end{tabular}
\end{table}

\begin{table}[htb]
\caption{MINRES numbers of iterations for Example \ref{ex:inlayer} with different values of $\beta$ and $\eps=10^{-3}$}\label{table:inlayerminres1}
\begin{tabular}{cllll|llll}
&\multicolumn{4}{c}{MG}&\multicolumn{4}{c}{BiCGSTAB}\\
\hline
\multirow{2}{*}{$k$}&\multicolumn{8}{c}{$\beta$}\\
\cline{2-9}\\[-2ex]
&$10^{-1}$&$10^{-2}$&$10^{-4}$&$10^{-8}$&$10^{-1}$&$10^{-2}$&$10^{-4}$&$10^{-8}$\\
\hline
$1$\ \vline&18&20&9&3&18&20&9&3\\
$2$\ \vline&19&22&12&4&19&22&12&4\\
$3$\ \vline&19&22&17&4&19&22&17&4\\
$4$\ \vline&21&23&21&4&21&23&21&4\\
$5$\ \vline&20&25&23&6&20&25&23&6\\
$6$\ \vline&21&25&25&7&20&25&25&7\\
$7$\ \vline&22&26&26&10&20&25&26&10\\
\hline
\end{tabular}
\end{table}

\section{Concluding Remarks}\label{sec:cncldremarks}




We  have proposed and analyzed discontinuous Galerkin methods for an optimal control problem constrained by a convection-dominated problem. Optimal estimates are obtained and an effective multigrid preconditioner has been developed to solve the discretized system. Numerical results indicate that our preconditioner is robust with respect to $\beta$ and $\eps$. However, theoretical justification of the robustness of our methods seems nontrivial. This will be investigated in a future project.

Our approach can also be easily extended to higher order DG methods assuming higher regularity of the solutions. One only needs to replace the projection estimates \eqref{eq:projesti} and \eqref{eq:projestiar} with
\begin{alignat*}{2}
     &\|z-\pi_hz\|_{\LT}+h\|z-\pi_hz\|_d\lesssim h^{l+1}\|z\|_{H^{l+1}(\O)}\quad&&\forall z\in V,\\
    &\|z-\pi_hz\|_{ar}\lesssim (\tau_c^{-\frac12}h^{l+1}+\|\bm{\zeta}\|_{0,\infty}^\frac12h^{l+\frac12})\|z\|_{H^{l+1}(\O)}\quad&&\forall z\in V,
\end{alignat*}
and proceed with the same argument as that of Theorem \ref{thm:dgesti}. Here the integer $l>1$ is the degree of the polynomials. We then obtain the following estimate which is similar to \eqref{eq:dgesti},
\begin{equation*}
\begin{aligned}
    & \trinorm{p-p_h}+\trinorm{y-y_h}\\
   & \lesssim C_\dagger\Big(\beta^\frac14(\eps^\frac12+\|\bm{\zeta}\|_{0,\infty}^{\frac12}h^{\frac12}+\tau_c^{-\frac12}h)h^{l}+h^{l+1}\Big)(\|p\|_{H^{l+1}(\O)}+\|y\|_{H^{l+1}(\O)}).
\end{aligned}
\end{equation*}

Our experiments have included BiCGSTAB($\ell$)  as building block for  our preconditioner for  comparison purposes. Although  the MINRES numbers of iterations by using BiCGSTAB($\ell$) were often the same as those of the multigrid operator, we emphasize that multigrid should still  be preferred in practice. Indeed, BiCGSTAB is a nonlinear solver because it 
 also depends on the right-hand side, so that the convergence of MINRES may be significantly affected by the BiCGSTAB solution accuracy. Moreover, BiCGSTAB depends on parameters such as a truncation and fill-in thresholds in its own ILU preconditioner. In contrast, multigrid may be used as a black box operator, and is an optimal $O(n)$ algorithm, where $n$ is the number of unknowns. We expect multigrid to outperform BiCGSTAB($\ell$) when $h \rightarrow 0$ in terms of computational time.

\section*{Acknowledgement}
This material is based upon work supported by the National Science Foundation under Grant No. DMS-1929284 while the authors were in residence at the Institute for Computational and Experimental Research in Mathematics in Providence, RI, during the Numerical PDEs: Analysis, Algorithms, and Data Challenges semester program.

Part of the work of VS was funded by the European Union - NextGenerationEU under the National Recovery and Resilience Plan (PNRR) - Mission 4 Education and research - Component 2 From research to business - Investment 1.1 Notice Prin 2022 - DD N. 104 of 2/2/2022, entitled ``Low-rank Structures and Numerical Methods in Matrix and Tensor Computations and their Application'', code 20227PCCKZ – CUP J53D23003620006. VS is member of the INdAM Research Group GNCS; its continuous support is gladly acknowledged.

\bibliographystyle{plain}
\bibliography{references}

\end{document}